\documentclass[5p, times]{elsarticle}

\usepackage{amsmath}
\usepackage{xcolor}
\usepackage{bm}
\usepackage{dcolumn,amsthm}
\newcommand\restr[2]{{
  \left.\kern-\nulldelimiterspace 
  #1 
  \littletaller 
  \right|_{#2} 
  }}

\usepackage{url}
\usepackage{chngcntr}
\usepackage{apptools}
\newcommand{\littletaller}{\mathchoice{\vphantom{\big|}}{}{}{}}
\renewenvironment{proof}[1][Proof]{\noindent\textit{#1. } }{\hfill$\square$}

\newtheorem{lemma}{\sc Lemma}[subsection]
\AtAppendix{\counterwithin{lemma}{subsection}}

\usepackage{amssymb}

\journal{-}

\usepackage{amsmath}
\usepackage{amsmath,amssymb}

\begin{document}

\begin{frontmatter}



\title{Pressure--Poisson Equation in Numerical Simulation of Cerebral Arterial Circulation and Its Effect on the Electrical Conductivity of the  Brain}


\author[inst1]{Maryam Samavaki\corref{cor1}}
\ead{maryamolsadat.samavaki@tuni.fi}
\cortext[cor1]{Corresponding author at: Sähkötalo building, Korkeakoulunkatu 3, Tampere, 33720, FI}

\affiliation[inst1]{organization={Mathematics, Computing Sciences, Tampere University},
            addressline={Korkeakoulunkatu 1}, 
            city={Tampere University},
            postcode={33014}, 
            country={Finland}}
\affiliation[inst3]{organization={Faculty of  Mathematics, K. N. Toosi University of Technology},
            addressline={Mirdamad Blvd, No. 470}, 
            city={Tehran},
            postcode={1676-53381}, 
            country={Iran}}
 \affiliation[inst2]{organization={Aix Marseille Univ}, addressline={ CNRS,  Centrale Marseille,  Institut Fresnel},city={Marseille},country={France}}  
\author[inst1,inst2]{Yusuf Oluwatoki Yusuf} 
 
\author[inst3]{Arash Zarrin nia}

\author[inst1]{Santtu S\"{o}derholm}

\author[inst1]{Joonas Lahtinen}

\author[inst1]{Fernando Galaz Prieto}

\author[inst1]{Sampsa Pursiainen}

\begin{abstract}
\textbf{Background and Objective:} This study considers dynamic modelling of the cerebral arterial circulation and reconstructing an atlas for the electrical conductivity of the brain. Electrical conductivity is a governing parameter in several electrophysiological modalities applied in neuroscience, such as electroencephalography (EEG), transcranial electrical stimulation (tES), and electrical impedance tomography (EIT). While high-resolution 7-Tesla (T) Magnetic Resonance Imaging (MRI) data allow for reconstructing the cerebral arteries with a cross-sectional diameter larger than the voxel size, electrical conductivity cannot be directly inferred from MRI data.
Brain models of electrophysiology typically associate each brain tissue compartment with a constant electrical conductivity, omitting any dynamic effects of cerebral blood circulation. Incorporating those effects poses the challenge of solving a system of incompressible Navier--Stokes equations (NSEs) in a realistic multi-compartment head model. However, using a simplified circulation model is well-motivated since, on the one hand, the complete system does not always have a numerically stable solution and, on the other hand, the full set of arteries cannot be perfectly reconstructed from the MRI data, meaning that any solution will be approximative. 

\noindent \textbf{Methods:} We postulate that circulation in the distinguishable arteries can be estimated via the pressure--Poisson equation (PPE), which is coupled with Fick's law of diffusion for microcirculation. To establish a fluid exchange model between arteries and microarteries, a boundary condition derived from the Hagen--Poisseuille model is applied. The relationship between the estimated \textit{volumetric blood concentration} and the \textit{electrical conductivity of the brain tissue} is approximated through Archie's law for fluid flow in porous media.

\noindent \textbf{Results:} Through the formulation of the PPE and a set of boundary conditions (BCs) based on the Hagen--Poisseuille model, we obtained an equivalent formulation of the incompressible Stokes equation (SE). Thus, allowing effective blood pressure estimation in cerebral arteries segmented from open 7T MRI data.

\noindent \textbf{Conclusions:} As a result of this research, we developed and built a useful modelling framework that accounts for the effects of dynamic blood flow on a novel MRI-based electrical conductivity atlas.  The electrical conductivity perturbation obtained in numerical experiments has an appropriate overall match with previous studies on this subject. Further research to validate these results will be necessary.



\end{abstract}



\begin{highlights}
\item Blood flow in cerebral arterial circulation is approximated by solving pressure--Poisson equation. 
\item An estimate for the volumetric blood concentration in microcirculation is obtained via Fick's law.
\item The effective electrical conductivity of the blood-tissue mixture is mapped using Archie's law.
\end{highlights}


\end{frontmatter}


\section{Introduction}
\label{sec:introduction}

This study focuses on the dynamic modelling of the cerebral arterial circulation \cite{caro2012mechanics}  and reconstructing an atlas for the electrical conductivity of the brain \cite{mai2015atlas}. Electrical conductivity is a governing parameter in several electrophysiological modalities targeting the brain, for example, electroencephalography (EEG) \cite{niedermeyer2004},  transcranial electrical stimulation (tES) \cite{herrmann2013transcranial}, and electrical impedance tomography (EIT) \cite{cheney1999electrical,moura2021anatomical,lahtinen2023silico}. 
Brain models of electrophysiology conventionally associate each head tissue compartment with a constant electrical conductivity  \cite{dannhauer2010,knosche2022eeg,ilmoniemi2019brain,demunck2012}, omitting the dynamic effects of cerebral circulation.
While high-resolution 7-Tesla (T) MRI data allows for the reconstruction of cerebral arteries with a cross-sectional diameter greater than the voxel size \cite{fiederer2016,svanera2021cerebrum}, electrical conductivity cannot be directly inferred based on MRI data.  
Advancements in dynamic conductivity modelling have addressed this critical gap in our understanding; in recent years, there has been growing interest in utilizing dynamic conductivity modelling to enhance our understanding of cerebral blood flow (CBF) regulation and its applications.
\\
 Incorporating blood flow effects is an important and timely topic that has most prominently been approached via rheo\-en\-ce\-pha\-lo\-grap\-hy \cite{bodo2018measurement},  EIT \cite{zhang2022pilot,ke2022advances}, and diffusion-weighted MRI \cite{lee2020extracellular} measurements. Alternatively, a dynamic model can be constructed {\em in silico} through blood flow simulation, as recently suggested in \cite{beraldo2020time,moura2021anatomical,lahtinen2023silico}. To this end, recent studies have, for instance, involved predicting the hydraulic conductivity of the capillary network, which is a critical factor in understanding microvascular dynamics~\cite{Sweeney2022.11.23.517681}.
 \\
 Modelling CBF poses the challenge of solving a system of Navier--Stokes equations (NSEs) numerically for an incompressible fluid flow in a complex structured blood vessel model, which has been demonstrated in several different contexts, e.g., in \cite{ Melis2017, blanco2017blood, blanco2014anatomically, zhu2015role}. Finding a numerical solution for incompressible NSEs includes, among other things: (1) the ability to incorporate incompressibility conditions into the system; (2) calculating pressure equivalently from flow velocity by assuming pressure boundary conditions (BCs) for the system; and (3) guaranteeing the numerical stability of the system \cite{prudhomme2002numerical}. In addition, since a full set of arteries cannot be perfectly reconstructed from the MRI data, only approximate solutions can be obtained. Hence, a simplified model is well-motivated. In particular, the model of \cite{moura2021anatomical} has been built upon a one-dimensional (1D) approximation of NSEs in the main arteries \cite{Melis2017}. In this study, we propose using the pressure--Poisson equation (PPE) \cite{pacheco2021continuous} in conjunction with Fick's law of diffusion for the microcirculation \cite{berg2020modelling,arciero2017mathematical} to estimate blood circulation in an arbitrary set of arteries distinguishable in MRI data.
\\
In incompressible flows, the pressure is in equilibrium with a time-varying divergence-free velocity field, and its gradient is a significant physical quantity since it represents a force per unit of volume. We derive PPE for a three-dimensional (3D) steady state blood flow and domain. We do not directly discretize the NSEs \cite{disNSE2013}; instead, we (i) model the blood flow in the arteries as the simplest possible system that can be derived from Stokes equation (SE), (ii) derive BCs assuming that the pressure gradient can be obtained from a Hagen--Poiseuille  flow in arterioles, which is in accordance with \cite{reichold2009vascular}, and then (iii) discretize and solve the resulting PPE system.
\\
 We utilize PPE to estimate pressure fields within the arterial domain by discretizing this system using a numerical method such as FE analysis while considering Neumann BCs. This process involves breaking down the arterial system into discrete elements, which enables us to solve the equation numerically. Subsequently, we utilize Fick's law to estimate the excess volumetric blood concentration within the microcirculation domain \cite{berg2020modelling, arciero2017mathematical}. We  discretize Fick's law using the FE method, to numerically solve it within the domain. Our hypothesis posits that there is a gradual reduction in this concentration with increasing distance from the arteries. This presumption is consistent with physiological principles, as it aligns with the process of diffusion, wherein blood  diffuses from areas characterized by higher concentrations to those with lower concentrations.  
\\
The parameters determining the BCs are based on the knowledge of average blood viscosity, cerebral blood flow rate, pressure, and diameter of the arterioles constituting an interface between arteries and capillaries \cite{caro2012mechanics}. The diffusion and decay parameter of the microcirculation model are, additionally, affected by experimental data of the microvessel density in different brain tissues \cite{kubikova2018numerical}. 
\\
The excess volumetric blood concentration plays a central role in determining the electrical conductivity $\sigma$ within the tissue. We incorporate this critical information into the electrical conductivity atlas, which is a key component of the proposed model. This atlas quantifies how the  electrical conductivity of the tissue varies in relation to the volumetric blood concentration. The relationship between microcirculation estimates and electrical conductivity is derived from Archie's law, a well-known and widely used two-phase mixture model. Archie's law allows us to determine the effective electrical conductivity of a mixture comprising fluid and porous media. Noteworthy studies \cite{j2001estimation,peters2005electrical,glover2000modified,cai2017electrical} have contributed to our understanding  of this relationship via electrical measurements.
As outlined in this study, the impact of the volumetric blood concentration on electrical conductivity remains significant within a specific distance, typically ranging from 10 to 20 mm, from the arteries. 
\\
To place our work in a broader scientific context, we explore the proposed model's connection with recent advancements in modelling dynamic blood flow along the capillary bed and their effect on the  electrical conductivity distribution of the brain tissue.
\\
Our results obtained with a finite element (FE) discretization of a multi-compartment head segmentation suggest that, given a high-resolution 7T MRI dataset, PPE, together with Fick's and Archie's laws, allows us to approximate blood pressure effects on the electrical conductivity  in the brain. We compare the results to a tissue-wise constant  distribution \cite{dannhauer2010} which provides the background model for Archie's law. Furthermore, we investigate potential future directions and applications of the proposed model, ultimately concluding this paper with an extensive discussion that emphasizes its importance and potential influence on future research. Our numerical implementation is openly available \cite{sampsa_pursiainen_2023_8200136}. 
\\


\section{Methods}
\label{sec: Theory} 
In this study, a domain for a human brain vasculature model is defined as the union  $\Omega \cup \hat{\Omega}$ of a microcirculation domain $\hat{\Omega}$ and a domain $\Omega$ composed of distinguishable arteries. 
In $\Omega$, the total pressure $p$ in the arteries is assumed to be of the form 
\begin{equation} 
   p = p^{\mathcal{(D)}}   + p^{\mathcal{(H)}}   \,, 
\label{total_pressure} 
\end{equation}
where $p^{\mathcal{(D)}}$ is a time-average of a dynamic arterial pressure and $p^{\mathcal{(H)}}$ is a hydrostatic venous pressure distribution following from a (constant) gravitational force field ${\bf f}$, blood density $\rho$ and position ${\bf x}$. Given a Riemannian metric $\bf g$ in $\Omega$, $p^{\mathcal{(H)}}({\bf x})$  can be expressed as 
\begin{equation} 
   p^{\mathcal{(H)}}({\bf x})=  \int_{\mathcal{C}({\bf x}, {\bf x}_0)} {\bf g}({\rho \,  \bf f}, \hbox{d} {\bf { r}})\,,
\label{hydrostatic_pressure}   
\end{equation}
where $\mathcal{C}({\bf x}, {\bf x}_0)$ is a geodesic, i.e., the shortest path on the surface, from a reference point ${\bf x}_0$ to ${\bf x}$ and $\hbox{d} {\bf r}$ its differential. The Riemannian metric utilized in formulations preserves local differences in shape and size across various head regions while maintaining the domain's shape across coordinate transformations. We include the Riemannian metric in our mathematical model for generality but do not explicitly discuss the curvature of the domain in this study. We solve PPE to obtain an approximation for $p$, after which volumetric blood concentration in $\hat{\Omega}$ is approximated using Fick's law of diffusion. Finally, electrical conductivity atlases are obtained based on the concentration via Archie's law. This section briefly reviews the theoretical grounds of PPE, Fick's law of diffusion, and Archie's law of two-phase electrical conductivity mixtures.


\subsection{Circulation in arteries}
In this study, blood is modelled as a non-ho\-mo\-ge\-ne\-ous, incompressible viscous fluid moving through blood vessels  
as a Newtonian flow with constant absolute dynamic viscosity of $\mu =  0.004$ Pa s, which can be considered a typical value in vessels with diameter from one to few millimeters with hematocrit between 45 and 60 \% \cite{pries1996biophysical,pries1992blood}. The corresponding 3D time-dependent Stokes equation is:
\begin{subequations}
\begin{align}
&\rho {\bf u}_{,t}-\mu{\bf L}{\bf u}+\nabla p=\rho \,{\bf {\bf f}}&\mathsf{in}\,\,\Omega \times [0, T]\,,
\label{eqn:line-1.1}
 \\
& \mathsf{div}({\bf u})=0&\mathsf{in}\,\,\Omega \times [0, T]\,,
\label{eqn:line-1.2}
  \\
&{\bf u}({\bf x};0)={\bf u}_0&\mathsf{on}\,\Omega \,,
\end{align}
\label{main-NSE}
\end{subequations}
where $\Omega $ is the physical domain of the problem and $[0, T]$ is the time domain. The blood velocity and pressure are defined as ${\bf u}={\bf u}({\bf x};t)$ and $p = p({\bf x}; t)$, respectively, in which ${\bf x}\in \Omega $ and $t\in\mathbb{R}^+$. The specification of \textit{the diffusion force}, denoted as ${\bf L}{\bf u}$, can be found in \ref{app:Diffusion model}. We identify $\rho({\bf x}; t)=\rho$ as a constant blood (mass) density, and the term ${\bf f}={\bf f}({\bf x}; t)$ on the right-hand side accounts for the possible action of \textit{external forces}. It is assumed that the initial velocity field ${\bf u}_0$ is divergence-free. The vector $\nabla p$ is a function of given velocity data and depends on the constitutive properties of blood,
$
\nabla p=\rho \,{\bf {\bf f}}-\rho {\bf u}_{,t}+\mu{\bf L}{\bf u}\,,
$
whose divergence leads to PPE. Derived from the momentum equation applying the incompressibility \eqref{eqn:line-1.2},  PPE of laminar flow is of the form
\begin{equation}
\begin{aligned}
    & \Delta p=\nabla\cdot (\rho \,{\bf {\bf f}})+2\mu \,\nabla\cdot ({\bf Ri}({\bf u}))&\mathsf{in}\,\,\Omega \,.
\end{aligned}
\label{pp}
\end{equation}
Because any harmonic function with a vanishing mean can be added to the above equation, it is clear that this does not define a unique $p$. As a result, we must consider the specific BC for the system \eqref{pp}. 
Now, we face two critical questions:  1) Can equation  \eqref{pp}  be used to calculate pressure $p$, and 2) does it imply incompressibility  \eqref{eqn:line-1.2}?  It seems that the answer to both questions is yes, if the divergence of Ricci curvature term $ \,\nabla\cdot ({\bf Ri}({\bf u}))$ is zero or small enough, i.e., if the geometry is locally flat.  Incompressibility for a static pressure field is implied by (\ref{pp}) under $\nabla\cdot ({\bf Ri}({\bf u})) = 0$, as justified in  \ref{app:incompressibile_steady_state}. 


\subsection{Pressure--Poisson equation}
\label{sec: Pressure-Poisson equatio}

We solve the pressure in system \eqref{main-NSE}  applying PPE and determining the proper BC based on that. Previously, FEs have been used to solve the pressure fields in the arteries driven by the so-called Neumann BCs, which are often sensitive and challenging to determine~\cite{2001-Ebbers, 2009-Ebbers}. It is theoretically sufficient to provide a BC for the system \eqref{main-NSE} by projecting \eqref{eqn:line-1.1} onto the boundary in either a normal or tangential direction. Thus, when Lemma \ref{div-lu} and the incompressibility requirement \eqref{eqn:line-1.2} are applied to the formula \eqref{eqn:line-1.1} together with the assumption that ${\bf f}$ is \textit{a constant gravity field} with $\nabla \cdot (\rho {\bf f}) = 0$. Utilizing formulas \eqref{total_pressure} and \eqref{hydrostatic_pressure}, the 3D-PPE model representing the dynamical aspect of the pressure field $p^{\mathcal{(D)}}$ takes the following form:
\begin{subequations}
\begin{align}
 &\Delta p^{\mathcal{(D)}} = 0 &\mathsf{in}\,\,\Omega\,,
\label{eqn:line-2.2}
\\
&{\bf g}(\nabla p^{\mathcal{(D)}}, {\vec{{\bf n}}})=-\zeta \lambda  ( p^{\mathcal{(D)}} -    {p}^{\mathcal{(B)}}) &\mathsf{on}
\,\,\partial \Omega \cap \partial \hat{\Omega}\,,
\label{eqn:line-2.4}
\end{align}
\label{main-PPE}
\end{subequations}
where parameter $\zeta$ is assumed to be constant, affecting the level of total blood flowing into the arterioles, and ${p}^{\mathcal{(B)}}$ is a distribution enforced on the boundary, determining the contribution of the incoming flow.  
Note that the domain $\Omega$ is flat with zero or close-to-zero curvature, the boundary condition is set on the common boundary $ \partial \Omega \cap \partial \hat{\Omega}$ of $ \Omega$ and $\hat{\Omega}$, and ${\vec{{\bf n}}}$ is the normal unit vector that is defined on the artery wall. The pressure drop on the boundary is denoted by $-{\bf g}(\nabla p, {\vec{{\bf n}}})=-\partial p/\partial\vec{{\bf n}}$ which is the inward normal derivative of the blood pressure and characterizes the behavior of the fluid near the boundary. 
The BC follows from the assumption that the flow is laminar in arterioles, which leads the blood through the boundary wall to the microcirculation domain. The total level of this flow is scaled by the pressure $p$ and $\lambda = \xi/\overline{\xi}$, which is defined as the ratio between the length density $\xi$ of microvessels \cite{kubikova2018numerical} per unit volume ($m^{-2}$), i.e., the total number of cross-sections per unit area, and the integral mean 
\begin{align*} 
\overline{\xi} = \frac{1}{|\partial \Omega|} \int_{\partial \Omega} \xi \,\hbox{d} \omega_{\partial \Omega}\,,
\end{align*}   
where $|\partial \Omega|=\int_{\partial \Omega}\,\mathrm{d}\omega_{\partial \Omega}$. We discuss the Hagen--Poiseuille model \cite{caro2012mechanics} motivating the BC in Section \ref{app:Boundary conditions} as a means of determining the BC. Assuming, for simplicity, a steady state flow,  the blood velocity can be approximated based on the solution of PPE, implying the following formula:
\begin{equation}
\mu \, \Delta_B{\bf u} = \nabla {p}^{\mathcal{(D)}} -\rho \, {\bf f}\,,
\label{u_artery}
\end{equation}
with ${\bf u} = 0$ on $\partial \Omega$  (\ref{app:incompressibile_steady_state}).  


\subsubsection{Variational form}

The equation determining blood flow pressure in the cerebral arteries is approximated numerically through FE discretization. We need to identify an appropriate variational formulation of PPE in order to continue with the FE discretization. Both the mathematical analysis and the numerical solution  are based on weak formulations. Integration by parts results in a weak general form of the above equations for blood pressure (\ref{main-PPE}) and velocity field (\ref{u_artery}) . 
\\
We assume that $p\in{\mathbb{W}}$ with 
${\mathbb{W}}=\mathrm{H}^1(\Omega)$ and ${\bf u}\in{\mathbb{V}}_0$ and ${\bf f}\in{\mathbb{V}}$, where ${\mathbb{V}}_0$ and ${\mathbb{V}}$ denote spaces of vector-valued functions in a physical 3D space, with $\mathbb{V}=[\mathrm{H}^1(\Omega)]^3$  and ${\mathbb{V}}_0 = [{ \mathrm{H}}^1_0(\Omega)]^3 \subset \mathbb{V}$. In other words, each of the three Cartesian components in ${\mathbb{V}}_0$ and ${\mathbb{V}}$ is in the \textit{Sobolev space} $\mathrm{H}^1(\Omega )$  of square-integrable ($\int_{\Omega} |{\bf u}|^2 \, \mathrm{d}\omega_{\Omega} < \infty$) functions with square integrable partial derivatives
\begin{align*}
\mathrm{H}^1(\Omega)&=\left\{u\in \mathrm{L}^2(\Omega)\,|\, \nabla u\in \mathrm{L}^2(\Omega )\right\}\quad\text{and}
\\
\mathrm{H}^1_0(\Omega)&=\left\{u\in\mathrm{H}^1(\Omega)\,|\, u|_{\partial \Omega}=0\right\}\,.
 \end{align*}
A variational form of \eqref{main-PPE} can be obtained by multiplying the equation \eqref{eqn:line-2.2} with a smooth enough test function $q\in {\mathbb{W}}$ and applying the divergence theorem. We arrive at the following variational problems:
\begin{itemize}
    \item[I.] Find ${p}^{\mathcal{(D)}}\in {\mathbb{W}}$ such that, for a smooth enough 
 test function $ q\in  {\mathbb{W}}$
\begin{equation}
 \begin{aligned}
 b( {p}^{\mathcal{(D)}}, q)=-\int_{\partial \Omega}q\,\zeta \lambda  ( {p}^{\mathcal{(D)}} -  {p}^{\mathcal{(B)}})\,\mathrm{d}\omega_{\partial \Omega}  \,.
\end{aligned}
 \label{PPE_var}
 \end{equation}
\end{itemize}
The continuous bilinear form $b:{\mathbb{W}}\times {\mathbb{W}}\rightarrow \mathbb{R}$ is defined as follows:
 \begin{align*}
b({p}^{\mathcal{(D)}}, q):= \int_{\Omega } {\bf g}(\nabla {p}^{\mathcal{(D)}}, \nabla q)\,\mathrm{d}\omega_{\Omega}\,.
\end{align*}
\begin{itemize}
     \item[II.] Find ${\bf u} \in {\mathbb{V}}_0$  such that, for a smooth enough 
 test function $ {\bf v} \in {\mathbb{V}}_0$
 \begin{equation}
     a({\bf u}, {\bf v}) = \int_\Omega {\bf g}(\nabla {p}^{\mathcal{(D)}},  {\bf v}) \, \hbox{d} \omega_\Omega  - \int_\Omega \rho \, {\bf g}({\bf f},  {\bf v}) \, \hbox{d} \omega_\Omega \,.
 \label{SE_var}
 \end{equation}
 \end{itemize}
The continuous bilinear form $a:{\mathbb{V}}_0 \times {\mathbb{V}}_0 \rightarrow \mathbb{R}$ is defined as follows:
\[
   a({\bf u}, {\bf v}) := \mu\int_{\Omega}{\bf g}(\nabla {\bf u}, \nabla {\bf v}) \,\mathrm{d}\omega_{\Omega}\,.
 \]


\subsection{Pressure boundary condition} 

The blood flows from $\Omega$ to the microcirculation domain $\hat{\Omega}$ through the 
 total cross-section area of the outlets of the arterioles on $\partial \Omega$. Arterioles are microvessels that  connect to arteries on one end and are continued by capillaries and thereon by venules in $\hat{\Omega}$. Most of the pressure decay in the blood flow, about 70 \% of the total pressure \cite{caro2012mechanics}, takes place in arterioles, which form a necessary transition zone for the blood pressure. In this study, we focus on modelling the blood flow within vessels and microvessels while excluding the circulation of interstitial fluid, which occurs outside the vessels.
\label{app:Boundary conditions}

\begin{figure}[h!]
\centering
\begin{minipage}{8.0 cm}
    \centering\includegraphics[width = 7.9 cm]{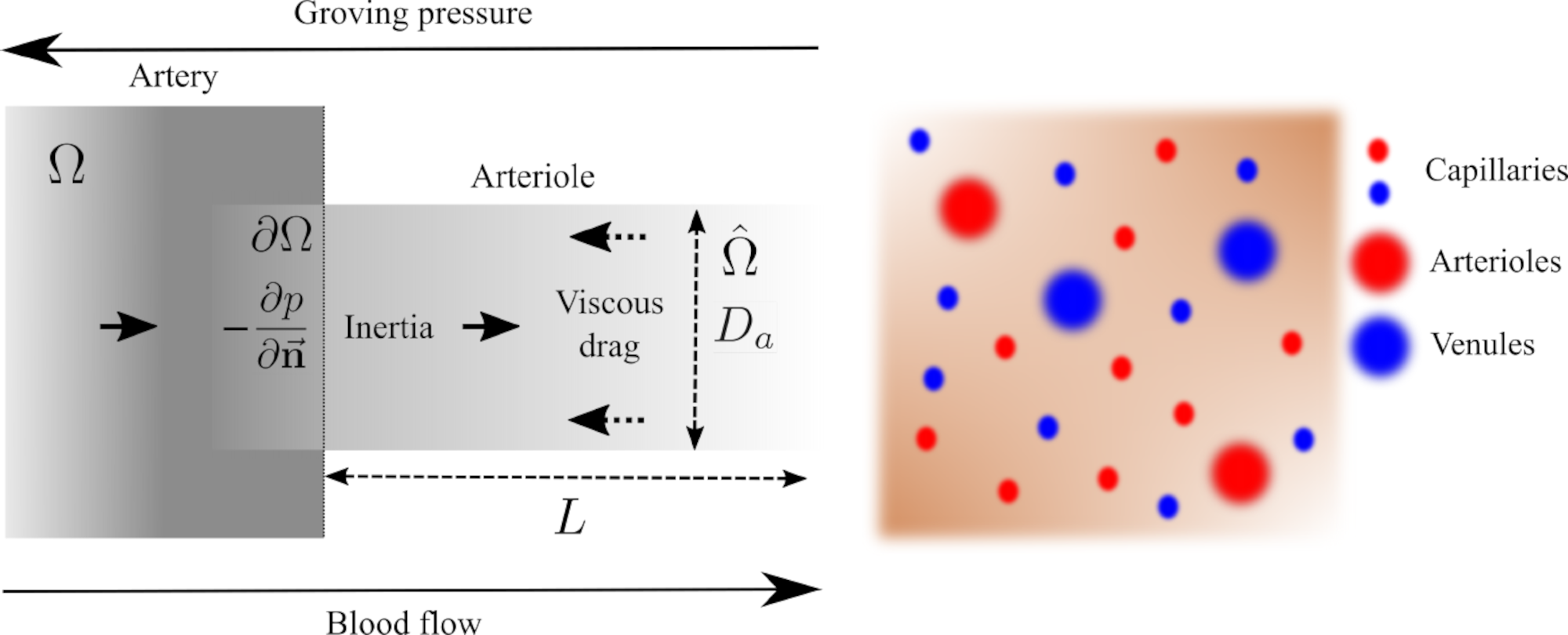}
\end{minipage}

\caption{\label{artery_arteriole_interface} Our model of the artery-arteriole interface shows the arteries in the domain $\Omega$ and the microcirculation domain $\hat{\Omega}$, where blood flows from the arteries into a network of arterioles, capillaries, and venules. {\bf Left:} The relationship between the flow rate and the normal derivative  ${\bf g}(\nabla p, {\vec{{\bf n}}})$ on $\partial \Omega$ of the pressure (with respect to the Riemannian metric) is determined by the Hagen--Poiseuille equation \cite{caro2012mechanics}, i.e., a laminar balance between the inertial force and viscous drag in a cylindrical tube with a diameter $D_a$ and a length $L$. {\bf Right:} A schematic of brain tissue cross-section with a set of blood vessels, including cross-sections of arterioles, capillaries, and venules. The length density of arterioles can be approximated by the total observed length density of microvessels  \cite{kubikova2018numerical}, the cross-section areas of individual microvessels, and the total cross-section area fractions between the different microvessel types \cite{tu2015human}. It is assumed that, due to the randomness of the vessel orientations, there is no orientational dependence in the length densities. }
\end{figure}


\subsubsection{Hagen--Poiseuille model} 
 \label{app:Hagen--Poiseuille model}
 
 
To obtain a value for $\zeta$, we use the Hagen--Poiseuille equation of laminar flow \cite{caro2012mechanics}, which, written for a single artery-arteriole interface (Figure \ref{artery_arteriole_interface}), is of the form \cite{caro2012mechanics}: 
\begin{equation}
\begin{aligned}
\vartheta p= \frac{8 \pi \mu L Q_a}{A_a^2}\,.
\end{aligned}
\label{HPE}
\end{equation}
Here, it is taken into account that $p$ stands for 
 a pressure drop between the inlet and outlet of an artery. As a result, it is scaled by the constant $\vartheta$, which determines the relative pressure drop in arterioles. In the equation \eqref{HPE}, $Q_a$ represents the blood flow rate through a single arteriole, $A_a = \pi/4\, D_a^2$ is its cross-sectional area with a diameter of an arteriole $D_a$, and $L$ represents its length. The value of $Q_a$ can be approximated as 
\begin{align*} 
    Q_a = \frac{Q}{ |\partial \Omega|  \, \overline{\xi}_a} \,,
\end{align*}
where $Q$ is the total flow through $\partial \Omega$ and  $\overline{\xi}_a$ is the average length  density of arterioles on $\partial \Omega$  (\ref{app:arteriole_density}). It follows that
\begin{equation}
L = \frac{|\partial \Omega| \, A_a^2 \,  \vartheta \overline{p} \,\overline{\xi}_a}{8\pi  \mu Q}\,, 
\label{arteriole_length}
\end{equation}
where $\overline{p}$ is a given average pressure. 
Since the inward normal derivative of $p$ represents force per surface area, we additionally scale the approximation following from Hagen--Poisseuille model (\ref{HPE}) locally by the relative area covered by the arterioles, i.e., the  product $ \xi_a  A_a $ between the length density $\xi_a$ and per vessel cross-sectional area $A_a$ of arterioles (Figure \ref{artery_arteriole_interface}), resulting in 
\begin{align*}
\zeta = \frac{8 \, \pi \, \mu Q}{  |\partial \Omega| \, A_a \,  \overline{p}  }\,. 
\end{align*}
Here it has been taken into account that $\lambda = \xi /\overline{\xi} = \xi_a /\overline{\xi}_a$ (\ref{app:arteriole_density}).


\subsection{Microcirculation model}

When blood flows from the arteries in $\Omega$ to the microcirculation domain $\hat{\Omega}$, it produces an excess blood volume concentration ${\bf c}={\bf c}({\bf x}; t)$ in comparison to the equilibrium state where the pressure in the head is constant. 
This can be counted as local blood supply upregulation, which varies the cross-sectional diameter  \cite{epp2020predicting}. To approximate ${\bf c}$, we apply advection-diffusion equation from Fick's second law of diffusion \cite{berg2020modelling, arciero2017mathematical, reichold2009vascular} and mass conservation. Fick's second law of diffusion and mass conservation have the following form: 
\begin{equation}
\begin{aligned}
& {\bf c}_{,t}=\nabla \cdot\big(-({\bf u}\otimes {\bf c})+{\bf J}\big)+\hat{{\bf c}}&\mathsf{in}\,\,\hat\Omega \! \times \! [0, T]\,,
\\
& \mathsf{div}({\bf u})=0&\mathsf{in}\,\,\hat\Omega \! \times \! [0, T]\,.
\end{aligned}
\label{DE.vessel}
\end{equation}
where ${\bf u}\otimes {\bf c}$ and ${\bf J}$ stand for the advective and diffusion flux, respectively. The flux ${\bf J}={\bf J}({\bf x}; t)$ is a vector pointing in the direction of movement, and the three-dimensional flux amplitude distribution $|{\bf J}|$ is proportional to the amount of blood flowing in the direction of ${\bf J}/|{\bf J}|$ per unit time. 
The following equation  is in accordance with Fick's first law under the assumption, that volumetric blood concentration and diffusive flow are proportionate to the tissue's relative microvessel density:
\begin{equation}
\label{fick}
  {\bf J}=-\varsigma \lambda \nabla{\bf c} \,.
\end{equation}
Here $\varsigma$ is an effective diffusion coefficient and is referred to as diffusivity.  In system \eqref{DE.vessel}, the term $\hat{{\bf c}}$ stands for a decay term. This results in the following governing differential system:
\begin{subequations}
\begin{align}
&{\bf c}_{,t}+{\bf u}\cdot\nabla {\bf c}-\varsigma\lambda \Delta_B {\bf c}=- \boldsymbol{ \varepsilon} {\bf c}&\mathsf{in}\,\,\hat\Omega \! \times \! [0, T]\,,
\label{DE.vessel_1.1}
\\
& \mathsf{div}({\bf u})=0&\mathsf{in}\,\,\hat\Omega \! \times \! [0, T]\,,
\label{DE.vessel_1.2}
\\
&  {\bf g}( \nabla c, {\vec{ \bf{n}}})= -\frac{1}{\varsigma \lambda}\,{\bf g}( {\bf J}, {\vec{ \bf{n}}}) &\mathsf{on}\,\, \partial \Omega \cap \partial \hat{\Omega},
\label{DE.vessel_1.3}
\end{align}
\label{DE.vessel_1}
\end{subequations}
where the Fickian and non-Fickian fluxes are represented by the terms ${\bf u}\cdot\nabla {\bf c}$ and $\varsigma\lambda\Delta_B {\bf c}$, respectively. The concentration is maximized on the boundary $ \partial \Omega \cap \partial \hat{\Omega}$ due to the inward flux over $ \partial \Omega \cap \partial \hat{\Omega}$  from $\Omega$ to $\hat{\Omega}$. The decay term $\boldsymbol{ \varepsilon} {\bf c}$ accounts for the flow from the microcirculation domain to the venous circulation system, which happens at a rate proportional to the coefficient $ \boldsymbol{\varepsilon}$. 
We consider a steady state solution, i.e., $ \lim_{t \to \infty} {\bf c}({\bf x}; t)=c({\bf x})$, in which concentration reaches a constant state or remains stable over time, $c({\bf x})=c$, and the flow's macrovelocity ${\bf u}$ vanishes.
The following simplified form results from applying the conservation of mass condition \eqref{DE.vessel_1.2} to the limiting equation \eqref{DE.vessel_1.1}
\begin{equation}
\begin{aligned}
 -\varsigma \lambda \Delta c+  {\bf \varepsilon} \, c=s\quad \mathsf{in}\,\,\hat\Omega\,,
\end{aligned}
\label{DE.vessel_2}
\end{equation}
where $s = \varsigma \lambda\, {\bf g}( \nabla c, {\vec{ \bf{n}}})\,$.


\subsubsection{Parameter estimation} 

To obtain an approximation for $\varsigma$ without taking into account the geometry, we rely on the Hagen--Poiseuille model for the arterioles (section \ref{app:Hagen--Poiseuille model}) with the additional assumptions that (i) the excess concentration is zero when the pressure is zero and that (ii) the concentration gradient is constant along the length of the microvessels. Under these assumptions, Fick's first law (\ref{fick}) for the flux 
passing $ \partial \Omega \cap \partial \hat{\Omega}$ is of the form 
\[
|{\bf J}| = \varsigma  \lambda \frac{\vartheta c}{ L} = \varsigma  \lambda \frac{\vartheta A_a \overline{\xi}_a}{ L}\,,
\]
where $A_a \overline{\xi}_a$ is the mean volume concentration of the blood and $L$ is the distance over which the concentration decreases from $\vartheta A_a \overline{\xi}_a$. Sustituting $|{\bf J}| =  \lambda Q_a \overline{\xi}_a$ and  the formula (\ref{arteriole_length}) for $L$, it follows that 
\[
\varsigma =  \frac{A_a \overline{p} }{8 \pi  \mu}\,.
\]
Parameter $\varepsilon$ can be obtained assuming that at each interior point in $\hat{\Omega}$ the blood in venules flows to a venous vessel, thereby exiting $\hat{\Omega}$. In a balanced state, the concentration loss caused by this flow equals the  density $| {\bf J} |$. We establish a balance by requiring that the sink amplitude ${\varepsilon}$ integrated over a radius $R$ sphere with a volume $\mathcal{V}_{\hbox{\scriptsize max}} = (4/3) \pi R^3$  of the largest element in the FE discretization matches an outward flux $|{\bf J}|$ integrated over the surface of the sphere, i.e., 
\[
\varepsilon \mathcal{V}_{\hbox{\scriptsize max}} = (4/3) \pi R^3 \varepsilon =  4 \pi R^2 | {\bf J} | \quad \hbox{or} \quad 
\varepsilon =   \varsigma \lambda \frac{ \vartheta }{ L  } \left( \frac{45 \pi}{\mathcal{V}_{\hbox{\scriptsize max}}} \right)^{1/3}\,.
\]

\subsubsection{Variational form}

The variational form of \eqref{DE.vessel_2} can be obtained in the same fashion as in the case of PPE. Multiplying \eqref{DE.vessel_2}
 with a smooth enough test function $h \in {\mathbb{W}}$ and applying the divergence theorem, we arrive at the following form: 
\begin{itemize}
  \item[] Find $c\in {\mathbb{W}}$ such that, for a smooth enough  test function $ q\in  {\mathbb{W}}$
\begin{equation}
 \begin{aligned}
   d( c, h) =  \int_{\partial \Omega}\varsigma \lambda \,h\, {\bf g}( \nabla c, {\vec{ \bf{n}}})\, \mathrm{d }\omega_{\partial \Omega}\,,
 \end{aligned}
\label{fick_variational}
\end{equation}
\end{itemize}
where the linear boundary term describing the incoming flow is on the left-hand side and  $\vec{\bf n}$  is the normal unit vector defined in the microcirculation domain. The continuous bilinear form $ d( c, h)$ is  defined as $d:{\hat{\mathbb{V}}}\times \hat{\mathbb{V}}\rightarrow \mathbb{R}$, where
 \begin{align*}
d(c, h)&= \int_{\hat{\Omega}}  \varsigma \lambda\,{\bf g}(\nabla c, \nabla h)\,\mathrm{d}\omega_{\hat{\Omega}} +  \int_{\hat{\Omega}}\varepsilon c\, h\,\mathrm{d}\omega_{\hat{\Omega}}\,.
\end{align*}


\subsection{Discretization}
\label{app:forward_model}

 We use the Ritz-Galerkin method \cite{braess2007finite} to discretize the PPE \eqref{main-PPE} and Fick's law of diffusion (\ref{DE.vessel_2}), whose solutions are assumed to be contained by the trial function spaces 
\begin{align*}
{\mathbb{V}_h}&=\mathsf{span}\big\{\psi^1,\dots,\psi^{n}\big\}
   \subset {\mathbb{V}}
   \\
 {\mathbb{W}_h}&=\mathsf{span}\big\{\varphi^1,\dots,\varphi^{m}\big\}
   \subset {\mathbb{W}}
   \\ {\hat{\mathbb{V}}_h}&=\mathsf{span}\big\{\phi^1,\dots,\phi^{m}\big\}
   \subset {\hat{\mathbb{V}}}\,, 
  \end{align*}
respectively. In each case, the discretization error is assumed to be orthogonal to the solution. 
Linear Lagrangian (nodal) basis functions are utilized in this context. Specifically, we have the sets $\{\psi^i\}_{i=1}^{n}$ and $\{\varphi^h\}_{h=1}^{m}$ supported in $\Omega$, as well as $\{\phi^h\}_{h=1}^{m}$ supported in $\hat{\Omega}$. These sets consist of piecewise linear functions that satisfy the conditions $\psi^i(x_j)=\delta^i_j$ for $i, j=1, \cdots, n$, $\varphi^h(x_k)=\delta^h_k$ for $h, k=1, \cdots, m$ at the FE mesh nodes of the finite element (FE) mesh in $\Omega$, and $\phi^h(x_k)=\delta^h_k$ for $h, k=1, \cdots, m$ at the nodes of $\hat{\Omega}$. Consequently, the velocity ${\bf u}\in {\mathbb{V}}$, pressure $p^{\mathcal{(D)}} \in \mathbb{W}$, and concentration $c \in \hat{\mathbb{V}}$ take  the following forms: 
\begin{align*}
 &u^{\ell}({\bf x})=\sum_{i=1}^{n} \psi^i({\bf x}) u^{\ell}_i\,, \quad f^{\ell}({\bf x})=\sum_{i=1}^{n} \psi^i({\bf x}) f^{\ell}_i\,,
 \\
 & {p}^{\mathcal{(D)}}({\bf x}) =  \sum_{i=1}^{m} \varphi^{i}({\bf x}) p_{i} \,,\quad c({\bf x}) \! = \! \sum_{i=1}^{m} \phi^{i}({\bf x}) c_{i}\,,
\end{align*}
\\
\sloppy 
for $\ell=1, 2, 3$. The coordinate vectors are denoted by {\setlength\arraycolsep{2pt} 
\begin{eqnarray*} 
{\bf u} & = & (u^1, u^2, u^3)=(\{u^1_i\}_{i=1}^n, \{u^2_i\}_{i=1}^n, \{u^3_i\}_{i=1}^n) 
\\ 
{\bf f} & = & (f^1, f^2, f^3)=(\{f^1_i\}_{i=1}^n,\{f^2_i\}_{i=1}^n, \{f^3_i\}_{i=1}^n) 
\\ 
 {\bf p }^{\mathcal{(D)}} & = & ({p}^{\mathcal{(D)}}, {p}^{\mathcal{(D)}}, {p}^{\mathcal{(D)}})=(\{p_i\}_{i=1}^m,\{p_i\}_{i=1}^m, \{p_i\}_{i=1}^m)
\\
{\bf c } & = &  (c_1, c_2, \ldots, c_m)\,.
\end{eqnarray*}

\subsubsection{Blood pressure and velocity in arteries}

The system of \eqref{PPE_var} is equivalent to the following Ritz-Galerkin discretized form: 
\begin{itemize}
     \item[I.] Find $p^{(\mathcal{D})}_h\in  {\mathbb{W}_h}$, such that, for all $ \varphi_h\in \mathbb{W}_h$
\begin{equation}
\begin{aligned}
b(p^{(\mathcal{D})}_h, \varphi_h)= -\int_{\partial \Omega}\varphi_h \,\zeta \lambda  ( p^{(\mathcal{D})}_h -  {p}^{\mathcal{(B)}})\,\mathrm{d}\omega_{\partial \Omega} 
\end{aligned}
\label{PPE_var_bar}
\end{equation}
\item[II.] Find ${\bf u}_h \in {\mathbb{V}}_{0,h}$  such that, for all $ \psi_h\in {\mathbb{V}}_{0,h}$
 \begin{equation}
     a({\bf u}_h, \psi_h) = \int_\Omega {\bf g}(\nabla {p}^{\mathcal{(D)}},  \psi_h) \, \hbox{d} \omega_\Omega  - \int_\Omega \rho \,{\bf g}( {\bf f},  \psi_h) \, \hbox{d} \omega_\Omega \,.
 \label{SE_var_discretized}
 \end{equation}
\end{itemize}
Here ${\mathbb{V}}_{0,h} \subset \mathbb{V}_h$ is obtained from  $\mathbb{V}_h$ by excluding the boundary degrees of freedom, i.e., basis functions with non-zero values on the boundary $\partial \Omega$. 
 \\
Equation (\ref{PPE_var_bar}) possesses a solution that satisfies the following equation:
\begin{equation}
({\bf K} + {\bf M}) \,  {\bf p}^{\mathcal{(D)}} = {\bf M} \, {\bf p}^{\mathcal{(B)}} \quad \hbox{with} \quad {\bf p}^{\mathcal{(B)}} = \hat{\bf p} + \overline{\bf p}. 
\label{d_ppe}
\end{equation} 
Here $\hat{\bf p}$ denotes a contribution of the incoming flow, normalized to a given pulse pressure $\hat{ p}$ and  $\overline{\bf p}$  is a normotensive  diastolic average, $\overline{\bf p}_i = \overline{p}$ for each  entry $i$. The matrices corresponding to equation \eqref{d_ppe} can be expressed as follows:
{\setlength\arraycolsep{2 pt}
\begin{align*}
&{\bf K}_{ij}   =   \int_{\Omega} \!  \varphi_{h, i }\,\varphi_{h, j}  \,  \hbox{d} \omega_{\Omega }
\\ 
&{\bf M}_{ij}  =  \int_{\partial \Omega} \! \zeta \lambda  \, \varphi_i \varphi_j  \, \hbox{d} \omega_{\partial \Omega},
\end{align*}}
where ${\bf K}=\left[{\bf K}_{ij}\right]_{m\times m}$ and ${\bf M}=\left[{\bf M}_{ij}\right]_{m\times m}$.

Following from the present modelling premises, we find an estimate for the systolic pressure distribution as a steady state solution, whose boundary restriction approximately satisfies ${\bf p}^{\mathcal{(D)}}_{|_{\partial \Omega}} = {\bf p}^{(\mathcal{B})}$, i.e., the boundary restriction of ${\bf p}^{\mathcal{(D)}}_{|_{\partial \Omega}}$ equals that of ${\bf p}^{(\mathcal{B})}$.   
Given an initial approximation ${\hat{\bf p}}^{(0)}$ of $\hat{\bf p}$, a recursive sequence of ${{\bf p}^{\mathcal{(D)}}}^{(1)}, {{\bf p}^{\mathcal{(D)}}}^{(2)}, \ldots$ is obtained by finding ${{\bf p}^{\mathcal{(D)}}}^{(k)}$ as a solution of (\ref{d_ppe}) corresponding to $\hat{\bf p}^{(k)}$, and setting $\hat{\bf p}^{(k+1)} = {{\bf p}^{\mathcal{(D)}}}_{|_{\partial \Omega}}$, for $k = 0, 1, 2, \ldots$. As a stopping criterion, we use the tolerance condition \[
\frac{\| {{\bf p}^{\mathcal{(D)}}}^{(K)} - {{\bf p}^{\mathcal{(D)}}}^{(K-1)}\|_2}{ \| {{\bf p}^{\mathcal{(D)}}}^{(K-1)} \|_2 } < \epsilon
\]
with $\epsilon = 0.01$ and fix ${{\bf p}^{\mathcal{(D)}}}^{(K)}$ as the final estimate of ${{\bf p}^{\mathcal{(D)}}}$.  The initial distribution ${\hat{\bf p}}^{(0)}$ is selected to be piecewise constant with  ${\hat{\bf p}}^{(0)}_i = \hat{p}$, if $i$ corresponds to one of two inlets (Figure \ref{fig:segmentation}) placed in the vicinity of the anterior and posterior end-points of Circle of Willis at the base of the brain, where the blood flow enters the brain, one in basilar artery and the other one in the junction of anterior cerebral and anterior communicating arteries. Other entries of  ${\hat{\bf p}}^{(0)}$ are set to zero. Two sources were applied to ensure balanced results; namely, asymmetry  of the flow in the Circle of Willis has been shown to extend to global scale \cite{zhu2015role}.
\\ 
The solution of equation \eqref{SE_var_discretized} satisfies  ${\bf  u}^{\ell} = 1/\mu\,{\bf L}^{-1} {\bf q}^{\ell}$, where the components of matrices ${\bf L}$ and ${\bf q}^{\ell}$ are obtained as follows:
{\setlength\arraycolsep{2 pt}
\begin{align*}
&{\bf L}_{ij}   =   \int_{\Omega} \mu \,\psi_{h, i}\,{\psi}_{h, j}\,\mathrm{d}\omega_{\Omega}
\\ 
&{\bf q}^{\ell}_{ij} =  \int_\Omega  p^{\ell}_{h}\,\varphi_{h, i} \,{\psi}_{j}  \, \hbox{d} \omega_\Omega   - \int_\Omega \rho\, f^\ell_i\,{\psi}_j\,\mathrm{d} \omega_\Omega\,,
\end{align*}}
where $\ell=1, 2, 3$. Matrix ${\bf L}$ can be obtained from ${\bf K}$ by excluding the boundary degrees of freedom from row and column indices. Namely,  the test function space $\mathbb{V}_h$ is linear and nodal akin to $\mathbb{W}_h$, but in  $\mathbb{V}_h$ the boundary degrees of freedom are set to zero due to the zero boundary condition for the velocity.}

\subsubsection{Volumetric Blood concentration in microcirculation}

The discretized diffusion problem related to the system of \eqref{fick_variational} can be formulated as follows:
\begin{itemize}
     \item[] Find $c_h\in  {\mathbb{W}_h}$, such that, for all $ \phi_h\in \mathbb{W}_h$
\begin{equation}
\begin{aligned}
 d( c_h, \phi_h) =  ( s, \phi_h)\,.
\end{aligned}
\label{fick_variational_var_bar}
\end{equation}
\end{itemize}
A numerical solution ${\bf c}$ of \eqref{fick_variational_var_bar} can be obtained via
\begin{equation}
\label{d_fick}
({\bf S} + {{\bf T}}) \, {\bf c} = {\bf w}\,,
\end{equation}
where
{\setlength\arraycolsep{2pt} 
\begin{align*}
{\bf S}_{ij} & = \int_{\hat{\Omega}}  \varsigma \lambda \,\phi_{h,i}\,  {\phi}_{h,j} \,  \hbox{d} \omega_{\hat{\Omega}}
\\ 
{\bf T}_{ij} & = \int_{\hat{\Omega}}   \varepsilon \,\phi_i \, {\phi}_j \,  \hbox{d} \omega_{\hat{\Omega}}
\\
{\bf w}_i & =\int_{ \hat{\Omega}}  s \,\phi_i \,  \hbox{d} \omega_{\hat{\Omega}}\,,
\end{align*}
and ${\bf w}=\begin{pmatrix} {\bf w}_1& {\bf w}_2 &\ldots {\bf w}_m    \end{pmatrix}^T$.



\subsection{Electrical conductivity of brain tissues}

To approximate how excess volumetric blood concentration perturbs the electrical conductivity distribution of the head, we apply Archie's law, a two-term linear combination of power functions that approximates the effective electrical conductivity $\sigma$ for a two-phase mixture of fluid and inhomogeneous medium \cite{j2001estimation,peters2005electrical,glover2000modified,cai2017electrical}. For a two-phase mixture, Archie's law is of the form
\begin{equation}
\label{archie}
\sigma =\sigma_m ( 1 - c )^\tau+\sigma_f {c}^{\beta} \text { with } \tau=\frac{\log \left(1-c^\beta \right) }{ \log (1-c) }\,, 
\end{equation}
where $\sigma_f$ and $\sigma_m$ denote conductivities of fluid and medium, respectively, and $\beta$ is so-called cementation factor \cite{j2001estimation,glover2000modified}, which for  the cerebral cortex is between $3/2$ and $5/3$ \cite{j2001estimation}. The lower and upper limits for $\beta$  follow from spherical  and cylindrical inhomogeneities, which in the cortex are represented by the somas and dendrites of the pyramidal cells, respectively. When substituted in the formula of Archie's law, $\beta = 3/2$ and $\beta = 5/3$ yield a lower and upper bound for the effective electrical conductivity, respectively.
\\
Alternatively, the effective electrical conductivity can be estimated
from above and below via Hashin--Shtrikman upper and lower
bound, defined as
\begin{align}
& \sigma^{+}=\sigma_f\left(1-\frac{3(1-c)\left(\sigma_f-\sigma_m\right)}{3 \sigma_f-c\left(\sigma_f-\sigma_m\right)}\right) \label{hsub} \\
& \sigma^{-}=\sigma_m\left(1+\frac{3 c\left(\sigma_f-\sigma_m\right)}{3 \sigma_m+(1-c)\left(\sigma_f-\sigma_m\right)}\right), \label{hslb}
\end{align}
respectively. Hashin-Shtrikman bounds have shown to be valid
for coated spherical inhomogeneities of all different sizes, filling
the space \cite{hashin1962variational}.

\subsection{Numerical experiments}
\label{Numerical simulationse}

\begin{table}[!ht]
    \centering
    \caption{Compartments of the head model segmentation were obtained using the FreeSurfer software suite \cite{fischl2012freesurfer} together with FieldTrip's \cite{oostenveld2011} segmentation interface, which has been built upon the functions of the SPM12 package \cite{ashburner2014spm12}. The vessel segmentation was performed using the Vesselness algorithm \cite{van2014scikit,frangi1998multiscale}. The piecewise constant background approximation of the electrical conductivity distribution $\sigma_m$ was based on \cite{dannhauer2010}. The subcortical active nuclei were associated with the conductivity of the grey matter \cite{rezaei2021reconstructing}. Vessel conductivity was chosen to match the blood conductivity \cite{gabriel1996compilation}. Skull and skin conductivity values were included in the head segmentation but not in the electrical conductivity atlases, since the segmented blood vessels were fully enclosed by the skull, inside the cranial cavity.}
    \begin{footnotesize}
    \begin{tabular}{llr}
    \hline 
        Compartment   &  Segmentation method & $\sigma_m$ (S m\textsuperscript{-1}) \\
        \hline
 Blood vessels & Vesselness &  0.70\\
    Grey matter & FreeSurfer &   0.33 \\
White matter  &  FreeSurfer  & 0.14 \\
Cerebellum cortex & FreeSurfer's Aseg atlas & 0.33 \\
Cerebellum white matter & FreeSurfer's Aseg atlas  & 0.14 \\
Brainstem & FreeSurfer's Aseg atlas & 0.33 \\ 
Cingulate cortex & FreeSurfer's Aseg atlas & 0.14 \\ 
Ventral Diencephalon & FreeSurfer's Aseg atlas & 0.33 \\
Amygdala & FreeSurfer's Aseg atlas & 0.33 \\ 
Thalamus &  FreeSurfer's Aseg atlas & 0.33 \\
Caudate & FreeSurfer's Aseg atlas & 0.33 \\  
Accumbens & FreeSurfer's Aseg atlas & 0.33 \\ 
Putamen & FreeSurfer's Aseg atlas & 0.33 \\
Hippocampus & FreeSurfer's Aseg atlas & 0.33 \\ 
Pallidum & FreeSurfer's Aseg atlas & 0.33 \\ 
Ventricles & FreeSurfer's Aseg atlas & 0.33 \\ 
       Cerebrospinal fluid (CSF) & FieldTrip-SPM12 & 1.79 \\    
              Skull & FieldTrip-SPM12 &  \\  
              Skin & FieldTrip-SPM12 &   \\  
                         \hline                     
    \end{tabular}
    \end{footnotesize}
    \label{tab:segmentation}
\end{table}

\begin{table}[!h]
    \centering
        \caption{The physical parameters applied in numerical simulations. Gravitational acceleration has been set to its average level and oriented parallel to the z-axis. The electrical conductivity of the blood $\sigma_f$ and the reference pressure were chosen according to  \cite{gabriel1996compilation} and \cite{blanco2017blood}, respectively. Blood density $\rho$,  $\mu$, total CBF $Q$, pressure decay in arterioles $\vartheta$, diameters $D_a$, $D_c$ and $D_v$ of arterioles, capillaries and venules (subtracting the total wall thickness, 2.0E-5, 2.0E-06 and 2E-6, respectively), and their relative total area fractions $\gamma_a$, $\gamma_c$ and $\gamma_v$, respectively, are based on the textbooks \cite{tu2015human,caro2012mechanics}. Microvessel density $\xi$ in cerebral and cerebellar grey and white matter (GM and WM), subcortical WM, and the brain stem was chosen according to the median values observed in \cite{kubikova2018numerical}.  The cementation factor estimates for spherical and cylindrical inhomogeneities approximating somas and dendrites of brain tissues, respectively, are based on \cite{j2001estimation}. Arteriole length has been obtained by substituting other parameter values in (\ref{arteriole_length}) with an appropriate correspondence to the values found in  literature \cite{caro2012mechanics}. }
    \begin{footnotesize}
    \begin{tabular}{lllr}
    \hline
        Property & Param.   & Unit & Value\\
        \hline
        Gravitation (z-component) & $f^{(3)}$ &  m s\textsuperscript{-2} & -9.81\\
         Electrical conductivity of blood & $\sigma_f$ & S m\textsuperscript{-1} & 0.70  \\
        Average diastolic pressure  & $\overline{p}$ & mmHg & 75  \\
        Pulse pressure  & $\hat{p}$ & mmHg & 40 \\
        Arteriole length & $L$ & mm & 0.4 \\
         Blood density & $\rho$ & kg m\textsuperscript{-3} & 1050  \\
         Viscosity & $\mu$ & m\textsuperscript{2} Pa s & 4.0E-03 \\ 
         Total CBF & $Q$ & ml min\textsuperscript{-1} & 750 \\
         Pressure decay in arterioles & $\vartheta$ & \% & 70 \\
         Arteriole diameter  & $D_a$ & m & 1.0E-05 \\
          Capillary diameter & $D_c$ & m & 7.0E-06 \\
           Venule diameter & $D_a$ & m & 1.8E-05 \\
        Arteriole total area fraction & $\gamma_a$ & \% & 25 \\
      Capillary area fraction & $\gamma_c$ & \% & 50 \\
    Venule area fraction & $\gamma_v$ & \% & 25 \\
    Microvessels in cerebral GM & $\xi$ & m\textsuperscript{-2} & 2.4E08 \\
    Microvessels in cerebral WM & $\xi$ & m\textsuperscript{-2} & 1.4E08 \\
    Microvessels in cerebellar GM & $\xi$ & m\textsuperscript{-2} & 3.0E08 \\
    Microvessels in cerebellar WM & $\xi$ & m\textsuperscript{-2} & 1.0E08 \\
    Microvessels in subcortical WM & $\xi$ & m\textsuperscript{-2} & 1.5E08 \\
     Microvessels in brainstem & $\xi$ & m\textsuperscript{-2} & 2.9E08 \\
      Cementation factor (spheres)  & $\beta$ & None & 3/2 \\
       Cementation factor (cylinders) & $\beta$ & None & 5/3 \\
       \hline                        
    \end{tabular}
    \end{footnotesize}
    \label{tab:parameters}
\end{table}


\subsubsection{Segmentation}

We performed numerical experiments using a realistic multi-compartment head model to assess PPE in combination with Fick's and Archie's laws in reconstructing an electrical  conductivity atlas of the brain. This head model was created using the open sub-millimeter precision Magneto Resonance Imaging (MRI) dataset\footnote{doi:10.18112/openneuro.ds003642.v1.1.0}  of CEREBRUM-7T \cite{svanera2021cerebrum}.  The dataset has been acquired using 7T magnetic flux density and, therefore, allows distinguishing the arterial vessels as a separate compartment, as shown in \cite{fiederer2016}. The FreeSurfer Software Suite \cite{fischl2012freesurfer}, FieldTrip's \cite{oostenveld2011} interface for the SPM12 surface extractor \cite{ashburner2014spm12}, and the Vesselness algorithm \cite{van2014scikit,frangi1998multiscale,fiederer2016} were applied to segment the arteries, i.e., domain $\Omega$. The other 16 brain compartments (Table \ref{tab:segmentation}), also enclosed by the skull, constituted $\hat{\Omega}$. Skin and  skull were not included in the electrical conductivity atlas since the domain of arterial vessels $\Omega$ was fully contained by the skull. The microcirculation domain $\hat{\Omega}$ included  microvessel-containing compartments: in addition to skin and skull, the cerebrospinal fluid (CSF) compartment and CSF-filled ventricles were excluded from $\hat{\Omega}$. 

\subsubsection{Vessel extraction}

The vessel extraction process was inspired by the work of Choi {\em et al.} \cite{choi2020cerebral} but is not entirely based on their proposed procedure. The Frangi filter was first applied to the MRI data slice-by-slice, and then the results were aggregated to produce the final arterial model. This process was performed in the following three steps: 

\begin{enumerate}
\item  Frangi's algorithm was applied to both the preprocessed INV2 and T1w slices of the dataset separately. We used the Scikit-Image \cite{van2014scikit} package's implementation of the Frangi method with different parameters for each slice. 
\item After applying the filter to a specific slice of the INV2 and T1w data, we created a mask by superposing  these two layers in an element-wise manner. The mask was binarized using a user-defined threshold level; every element with value less than this threshold was set to zero and, otherwise, to one. 
\item  The segmented cerebral vessels were obtained by iterating the previous steps through an axis of the MRI image and aggregating the results. In order to reduce noise, aggregation was performed separately for sagittal, axial, and coronal slices using the following scoring scheme: if a voxel was detected as a vessel in two or three of the results, it was considered a vessel in the final vessel mask; otherwise, it was neglected.
\end{enumerate}


\subsection{Numerical simulations}

The numerical simulations were performed using the open Zeffiro Interface \cite{sampsa_pursiainen_2023_8200136, he2019zeffiro} (ZI) toolbox. Solvers for PPE, Fick's law, and Archie's law were implemented as Matlab codes and included in ZI\footnote{\url{https://github.com/sampsapursiainen/zeffiro_interface}}. Using ZI, the volume of the head segmentation was discretized by a tetrahedral FE mesh of 6.4 M nodes and 32 M elements, corresponding approximately to 1 mm overall resolution. Of these,  $\Omega$ contained 0.15 M nodes and 0.54 M tetrahedra, and  $\hat{\Omega}$ 2.4 M nodes and 11 M tetrahedra. Other relevant parameter values can be found in Table \ref{tab:parameters}. 
\\
After solving the discretized pressure in $\Omega$ from \eqref{d_ppe}, the excess blood concentration $c$ in $\hat{\Omega}$ was obtained by solving (\ref{d_fick}). Archie's law (\ref{archie}) was evaluated using the excess concentration $c$ and two alternative cementation factors $\beta=5/3$ and $\beta=3/2$  corresponding to cylindrical and spherical tissue inhomogeneities and constituting a lower and upper bound approximation for the electrical conductivity, respectively. In addition, the Hashin--Shtrikman lower and upper bounds (\ref{hslb}) and (\ref{hsub}) were evaluated as an alternative approximation. 
\\
As a result, altogether five different electrical conductivity atlases were obtained: one corresponding to the piecewise constant background distribution and four effective electrical conductivity atlases following from the different mixture models. To examine differences between the background ${ \sigma}_{\hbox{\scriptsize bg}}$ and effective ${ \sigma}_{\hbox{\scriptsize eff}}$ distributions, the following difference measures (\%) were evaluated: 
{\setlength\arraycolsep{2 pt} 
\begin{eqnarray}
   \hbox{RDM} & = & 100 \left\| \frac{{ \sigma}_{\hbox{\scriptsize eff}}}{\| { \sigma}_{\hbox{\scriptsize eff}} \|_1} -  \frac{{ \sigma}_{\hbox{\scriptsize bg}}}{\| { \sigma}_{\hbox{\scriptsize bg}} \|_1}   \right\|_1\,, 
   \\
   \hbox{MAG} & = & 100 \frac{ \|  { \sigma}_{\hbox{\scriptsize eff}} \|_1}{\| { \sigma}_{\hbox{\scriptsize bg}} \|_1} -  100\,, 
   \\
   \hbox{PRD} & = & 100 \frac{ |  {\sigma}_{\hbox{\scriptsize eff}} -{\sigma}_{\hbox{\scriptsize bg}} | }{\| { \sigma}_{\hbox{\scriptsize bg}} \|_\infty}\,.  
\end{eqnarray}}
Of these, RDM (relative difference measure) evaluates the overall relative difference between normalized distributions ${ \sigma}_{\hbox{\scriptsize eff}}$ and ${ \sigma}_{\hbox{\scriptsize bg}}$, MAG (magnitude measure) shows the average amplitude of ${ \sigma}_{\hbox{\scriptsize eff}}$ compared to ${ \sigma}_{\hbox{\scriptsize bg}}$, and PRD (pointwise relative difference) is the difference between  ${ \sigma}_{\hbox{\scriptsize eff}}$ and ${ \sigma}_{\hbox{\scriptsize bg}}$ in relation to $\| { \sigma}_{\hbox{\scriptsize bg}} \|_\infty$ for each point.


\section{Results}
\label{sec: Results}

This section describes the results of our four-phase numerical simulations, where (i) the brain model was segmented to obtain $\Omega$ and $\hat{\Omega}$ as well as an estimate for the background tissue concentration; (ii) the  blood pressure $p$ and velocity ${\bf u}$ following from the present PPE model were found; (iii) Fick's law was applied to find an estimate $c$ for excess blood concentration in the microcirculation domain $\hat{\Omega}$; finally, (iv) effective electrical conductivity atlases were reconstructed by estimating the effect of the excess blood concentration on the background distribution via Archie's law and Hashin--Shtrikman bounds. The results of the numerical experiments have been included in Figure \ref{fig:segmentation} showing the head segmentation results; Figure \ref{fig:PPE_Fick_solutions} with sagittal, axial, and coronal illustrations of the pressure and concentration distributions obtained as numerical solutions of PPE and Fick's law; Table \ref{tab:RDM_and_MAG} including RDM and MAG for the estimates of background and effective electrical conductivity atlases obtained via Archie's model;  histograms  showing the value distributions of the blood pressure, velocity, and volumetric blood concentration (Figure \ref{fig:histograms_1})  and the PRD of the electrical conductivity distribution (Figure \ref{fig:histograms_2}); and Figure \ref{fig:conductivity_atlases} visualizing effective conductivity atlases for sagittal, axial, and coronal projections.

\begin{figure}[h!]
\centering
\begin{minipage}{8.0 cm}
    \centering
    \includegraphics[width = 7.9 cm]{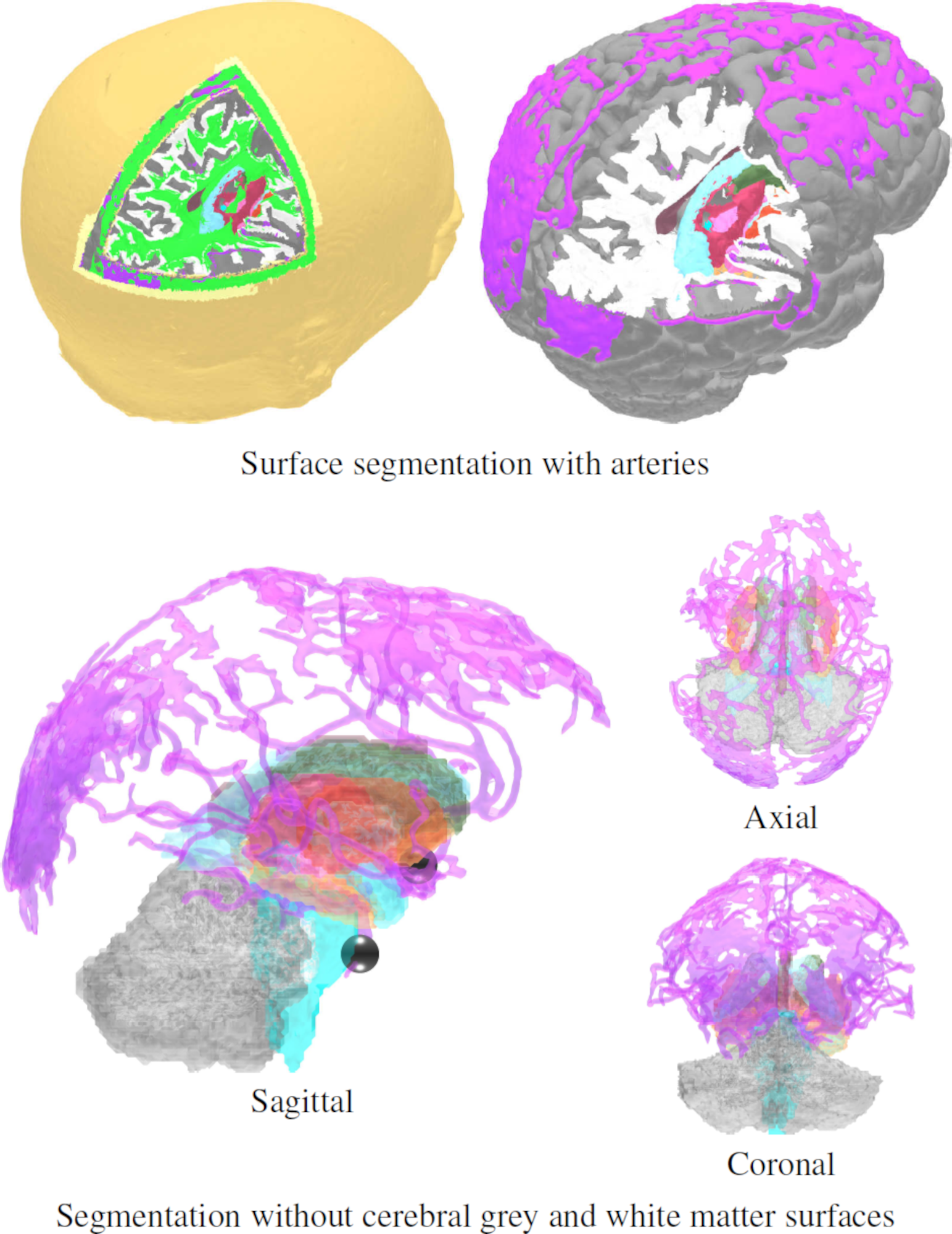}
\end{minipage}

\caption{{\bf Top:} A surface segmentation was obtained for the subject sub-045 of the open CEREBRUM-7T dataset\textsuperscript{1} \cite{svanera2021cerebrum} containing 7-Tesla (T) MRI data. Surface meshes (Table \ref{tab:segmentation})  were extracted using the segmentation routines of the FreeSurfer software suite \cite{fischl2012freesurfer} and FieldTrip's \cite{oostenveld2011} interface for SPM12's surface extractor \cite{ashburner2014spm12}. The cerebral arterial vessels shown on the right (violet) were segmented using Frangi's Vesselness algorithm \cite{van2014scikit,frangi1998multiscale}  as suggested in \cite{fiederer2016}. The clipping planes correspond to the sagittal, axial, and coronal surface slices shown in this study. {\bf Bottom:} Sagittal, axial, and coronal projections of the segmentation without cerebral grey and white matter surfaces, showing the arteries and how they integrate with the subcortical structures. Two spherical surfaces show the locations of two 10 mm diameter inlets placed in the vicinity of the anterior and posterior end-points of the Circle of Willis at the base of the brain. The locations  correspond to the  basilar artery  and the junction of anterior cerebral and anterior communicating arteries.   }
\label{fig:segmentation}
\end{figure}

\begin{figure}[h!]
\centering
\begin{minipage}{8.0cm}
    \centering
    \includegraphics[width = 7.9cm]{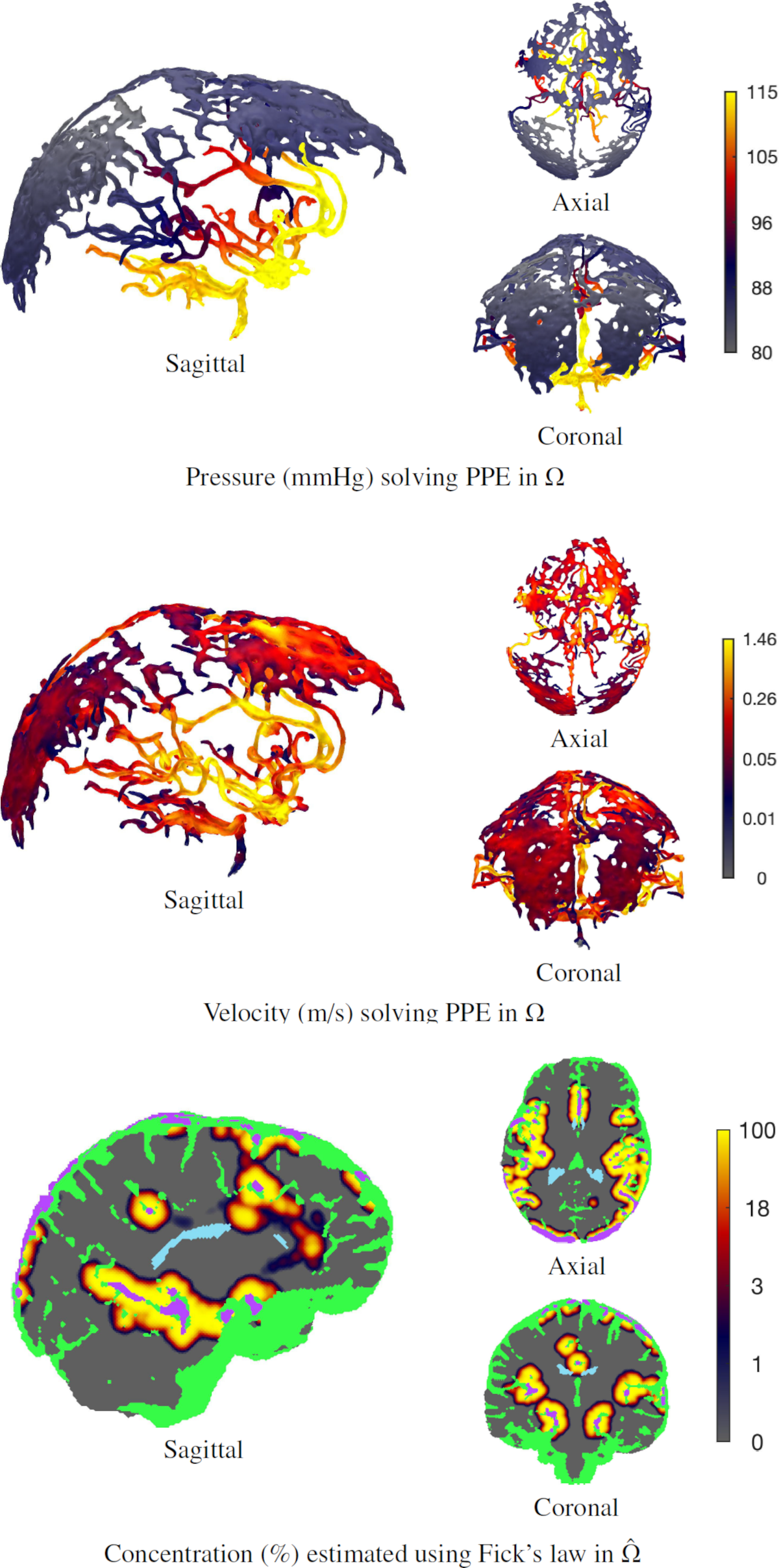}
\end{minipage}

\caption{Approximations  of systolic blood pressure, velocity, and volumetric blood concentration illustrated on a logarithmic color scale {\bf Top:} Sagittal, axial, and coronal views of the pressure distribution (mmHg) solving PPE  in domain $\Omega$. The large subcortical arteries can be observed to have a higher pressure up to 115 mmHg, compared to the smaller arteries of the cortical branches, for which the value range extends down to 80 mmHg. Values larger than  $q_{75} + 1.5 (q_{75} - q_{25})$, where $q_{25}$ and  $q_{75}$ denote the 25th and 75th percentiles, respectively, have been excluded as outliers.
{\bf Center:} The velocity distribution  (m/s) in $\Omega$. The greatest values extend up to or close to 1.46 m/s. The anterior and middle cerebral arteries can be observed to have overall greater velocities than the posterior cerebral arteries or the arteries with smaller cortical branches. Outliers larger than  $q_{75} + 1.5 (q_{75} - q_{25})$ have been excluded.
{\bf Bottom:} Sagittal, axial, and coronal surface cuts of the estimated excess blood concentration  (\%) in the microvessel domain $\hat{\Omega}$ as predicted by Fick's law. The greatest values are obtained in the vicinity of the arterial vessel boundary $\partial \Omega$, i.e., the boundaries of the violet subdomains.  The concentration decays to zero within 10--20 mm of a $\partial \Omega$. The visible structures not included in $\hat{\Omega}$ include the arterial vessels, i.e.\ $\Omega$, ventricles (blue), and cerebrospinal fluid (green).}
\label{fig:PPE_Fick_solutions}
\end{figure}

\begin{table}[h!]
    \centering
    \caption{Relative difference measure (RDM) and magnitude (MAG) difference for the  mixture models applied in this study.}
    \label{tab:RDM_and_MAG}
    \begin{footnotesize}
    \begin{tabular}{llrrr}
\hline 
Mixture model &  Type & RDM & MAG \\ 
\hline
Archie's law for spheres & Lower bound &  11.19 & 1.65 \\
Archie's law for cylinders & Upper bound & 11.32 & 1.69 \\
 Hashin--Shtrikman & Lower bound & 11.29 & 1.72 \\
 Hashin--Shtrikman & Upper bound & 11.85 & 1.86  \\
 \hline
\end{tabular}
\end{footnotesize}
\end{table}
\begin{figure}[h!]
\centering
\begin{minipage}{8.0cm}
    \centering
    \includegraphics[width = 7.9cm]{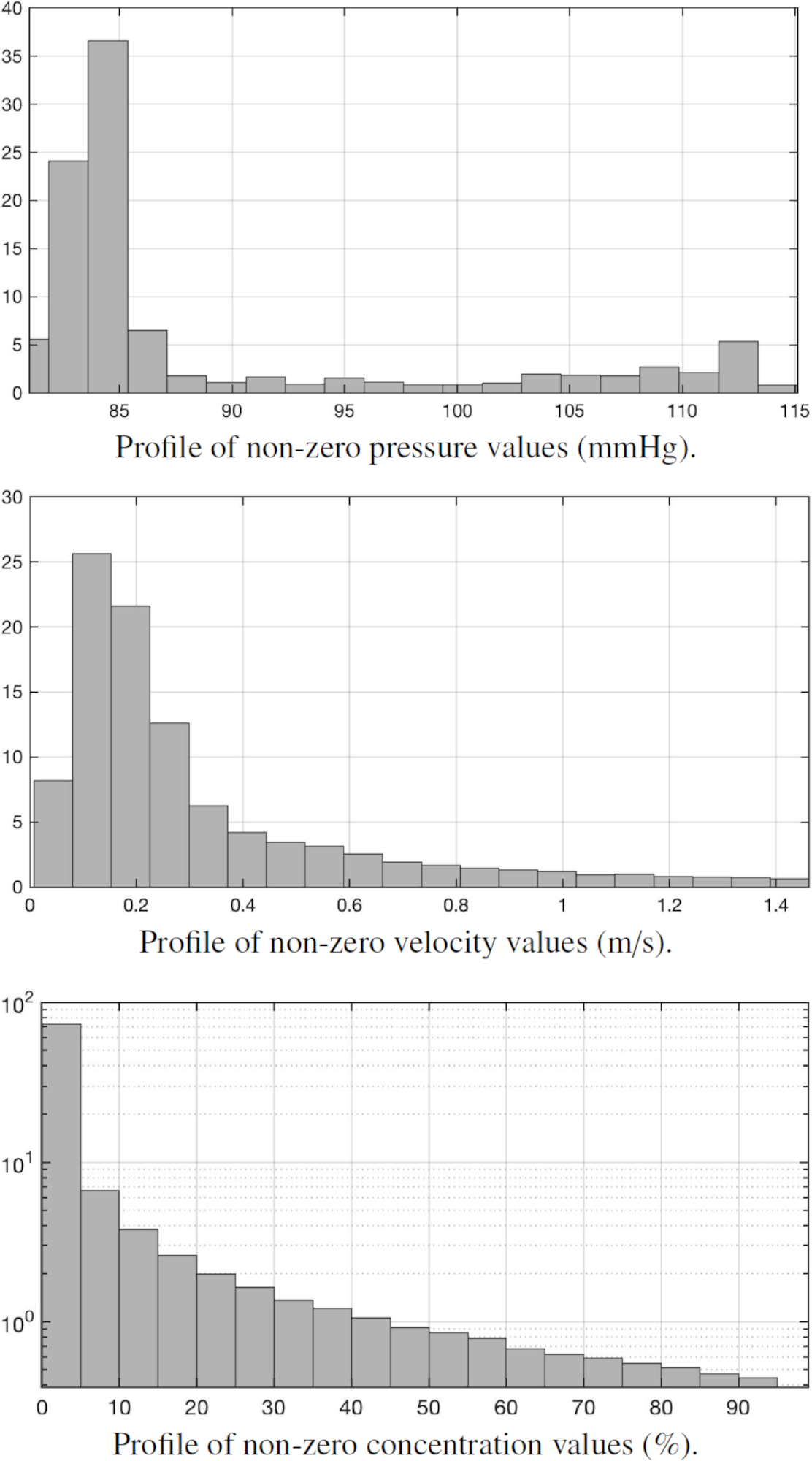}
\end{minipage}
\caption{Histograms for the distributions of blood pressure, velocity, and volumetric concentration. The horizontal axis shows the value of the pressure (top), velocity (center), or concentration (bottom), and the vertical one shows the corresponding volume fraction of $\hat{\Omega}$ (\%) on a logarithmic scale. Each visualization includes 20 bars. For pressure and velocity, outliers larger than  $q_{75} + 1.5 (q_{75} - q_{25})$, where $q_{25}$ and  $q_{75}$ have been excluded, the concentration is shown for the values above 0 and below 100 \%.  }
\label{fig:histograms_1}
\end{figure}

\begin{figure}[h!]
    \centering
\begin{minipage}{8.0cm}
    \centering
    \includegraphics[width = 7.9cm]{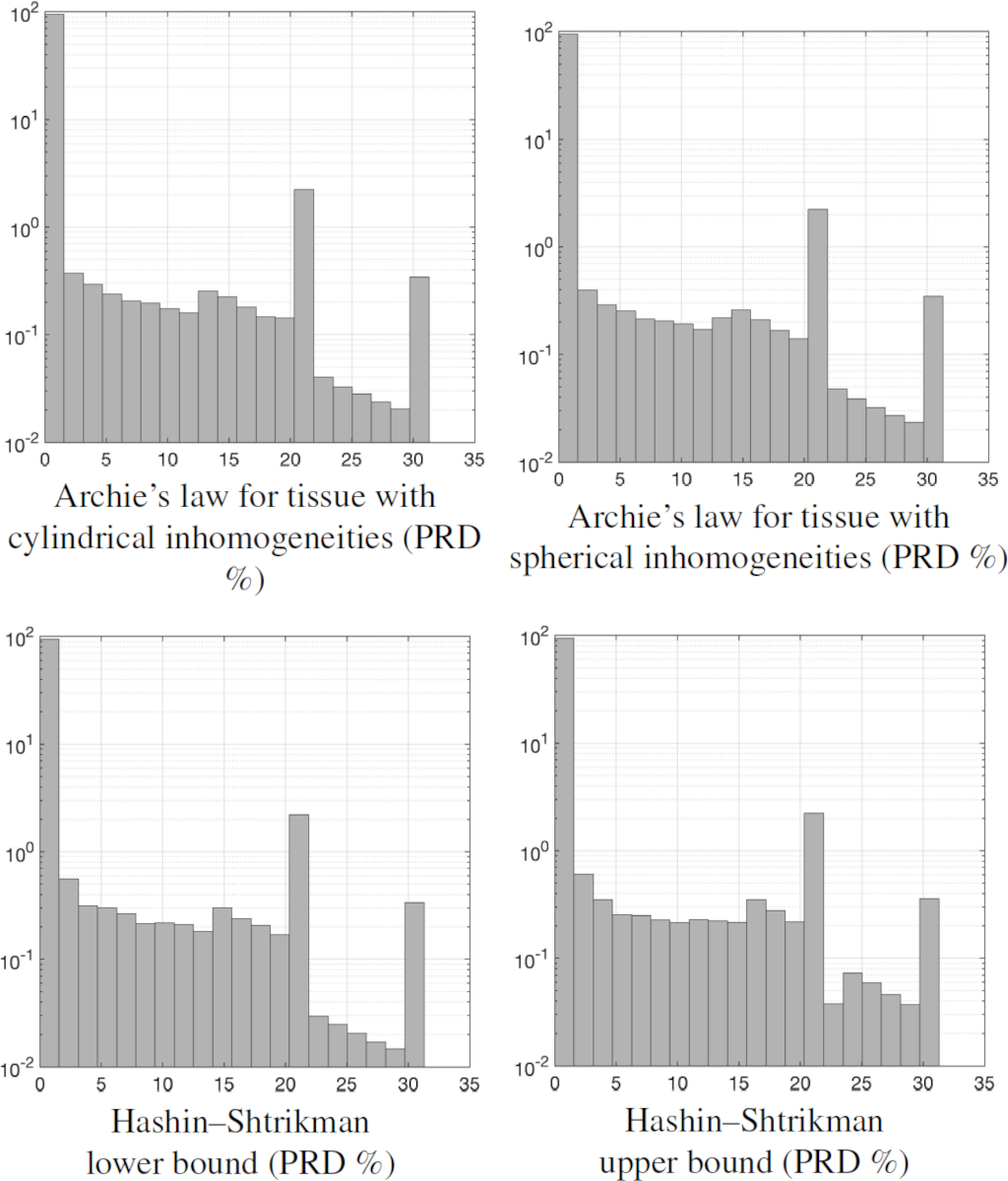}
\end{minipage}
    
    \caption{ Histograms showing the distribution of the relative difference PRD (\%) between the piecewise constant background and the approximated effective electrical conductivity in those parts of $\hat{\Omega}$, where PRD exceeds 0.1 \%. The horizontal axis shows the value of PRD, and the vertical one shows the volume fraction (\%) of the corresponding computing domain on a logarithmic scale. Each visualization includes 20 bars.   }
    \label{fig:histograms_2}
    \end{figure}

\begin{figure*}[h!]
\centering
\begin{minipage}{16.0cm}
    \centering
    \includegraphics[width = 15.5cm]{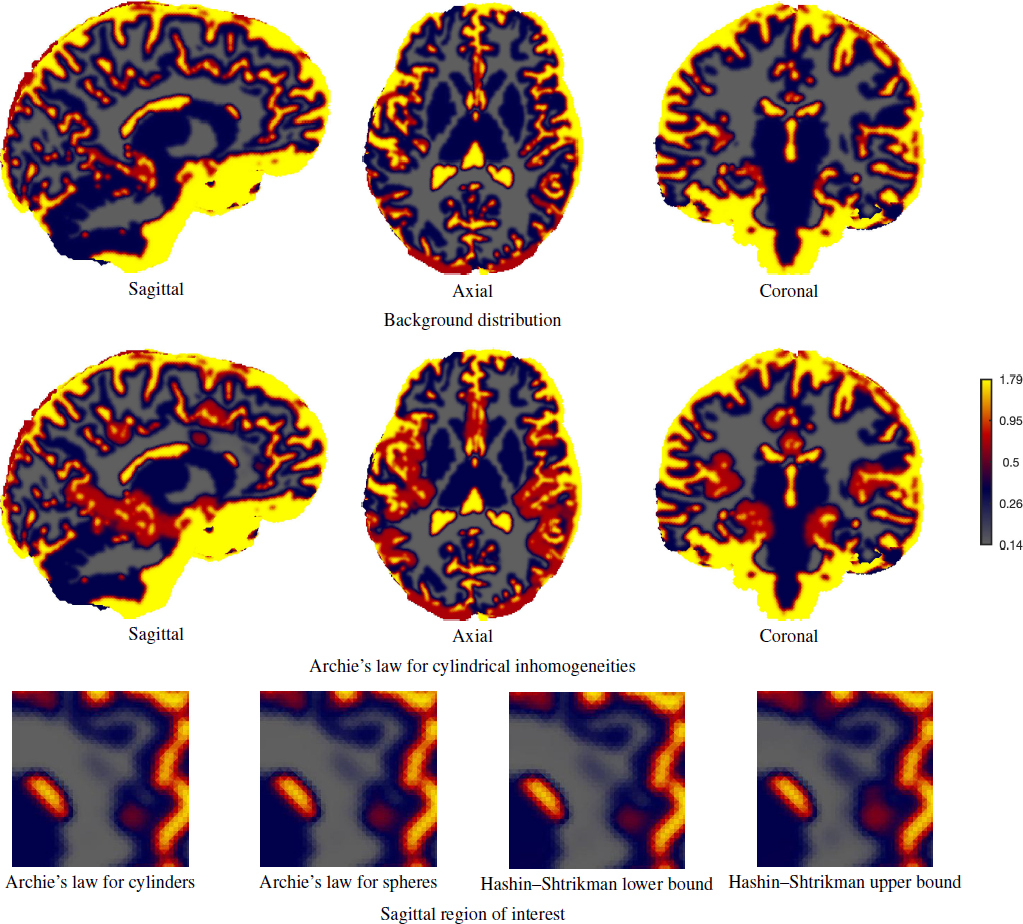}
\end{minipage}

    \caption{   {\bf Top row:} Sagittal, axial, and coronal visualization of the piecewise  constant background distribution $\sigma_m$ in which each compartment corresponds to a constant electrical conductivity  given in  Table \ref{tab:segmentation}. Compared to the background, each effective electrical conductivity distribution the conductivity in the vicinity of the arteries is emphasized. 
    \label{fig:background_conductivity} {\bf Center row:}: Sagittal, axial, and coronal illustration of the effective electrical conductivity atlas obtained with Archies's law assuming that the tissue inhomogeneities are  cylindrical.  {\bf Bottom row}:  A comparison between the different effective electrical conductivity atlases in a sagittal region of interest. The differences can be observed to be minor. Hashin--Shtrikman upper bound shows more spread effects of the blood flow than the other estimates.
 }
    \label{fig:conductivity_atlases}
\end{figure*}


\subsection{Phase (i): segmentation}

The segmentation obtained (Figure \ref{fig:segmentation}) shows that the Vesselness algorithm found a connected set of arteries inside the skull. Finding a connected set was considered important for the continuity of the PPE solution. Therefore, it was prioritized in the segmentation process over resolution, which led to an overlap of the tightly packed vessel bundles following from cortical branches. This geometrical distortion was considered unavoidable, due to the relatively small diameter of the cerebral arteries compared to the 1 mm resolution of the discretization.

\subsection{Phase (ii): PPE}

As shown by Figures \ref{fig:PPE_Fick_solutions} and \ref{fig:histograms_1}, excluding the outliers, the total pressure distribution $p$ varies between 80 mmHg and 115 mmHg in the artery domain $\Omega$. The values are the greatest at the base of the brain, close to the  basilar artery, the deepest vessel in $\Omega$, which is oriented nearly vertically in front of the brainstem. The pressure gradually decreases when moving towards the cerebral cortex, where branched, overlapping structures are dominant. This result is expected since the total vessel area gradually increases as the branching occurs. Thus, the pressure gradually decreases as the blood flows towards smaller vessels, eventually entering the microcirculation domain $\hat{\Omega}$. 
\\
The blood velocity profile varies between 0 and 1.46 m/s, excluding the outliers (Figures \ref{fig:PPE_Fick_solutions} and \ref{fig:histograms_1}). 
 The greatest values were observed in anterior and middle cerebral artery which had overall greater velocities than posterior cerebral artery or the smaller arteries  in cortical branches.

\subsection{Phase (iii): excess blood concentration}

The excess blood concentration estimate $c$ obtained via Fick's law (Figure \ref{fig:PPE_Fick_solutions}) expectedly decays when moving away from the arteries in $\Omega$ that bring blood into the microcirculation domain $\hat{\Omega}$. The amplitude of $c$ vanishes at a distance of 10–20 mm from the arteries. Consequently, the effect of the concentration on the electrical conductivity atlases is, within the present model, limited to this distance.

\subsection{Phase (iv): effective electrical conductivity atlases}
As shown by Table \ref{tab:RDM_and_MAG}, the reconstructed  atlases differ overall by 1.65--1.86 \% and 11.19--11.85 \% with respect to the MAG and RDM, respectively. The major part of the differences is limited to a few percent of the volume fraction (Figure \ref{fig:histograms_1}), which is obvious based on the excess concentration estimates obtained in the third phase and is verified by PRD, showing that locally the largest differences are approximately 30 \% with respect to the maximum (1.79 S/m) of the background distribution. Those can be related to regions close to the vessels where the excess blood concentration is close to one.  
Compared to the other atlases, the Hashin--Shtrikman upper bound yields a greater volume fraction of PRD  values slightly below 30 \%. Spatial differences between atlases corresponding to different mixture models are minor, which can be observed based on Figure \ref{fig:conductivity_atlases}.


\section{Discussion}
This study demonstrated that a simplification of the Stokes equation (SE), namely the pressure--Poisson equation (PPE) \cite{pacheco2021continuous} allows for the estimation of the blood pressure in cerebral arteries segmented from open 7T MRI data \cite{svanera2021cerebrum}. We introduced a boundary condition (BC) based on the Hagen--Poisseuille model \cite{caro2012mechanics} to bind PPE with the governing physical parameters of CBF, particularly the microvessel diameters \cite{caro2012mechanics} and densities \cite{kubikova2018numerical}. Through the formulation of the PPE and the BC, we obtained an equivalent formulation of the incompressible SE.
\\
Based on the solution of PPE, we estimated the excess volumetric blood concentration in microvessels caused by the pressure using Fick's law \cite{berg2020modelling,arciero2017mathematical}, the parameters of which were likewise obtained via the Hagen--Poisseuille model. Finally, the effect of the excess concentration on the brain tissues was approximated using Archie's law as well as the upper and lower bounds of Hashin and Shtrikman \cite{j2001estimation}. Our four-phase modelling process (i) first generates a multi-compartment FE mesh and a piecewise conductivity atlas of the head, then (ii) finds a solution for PPE and (iii) Fick's law, and finally, (iv) reconstructs an atlas. 
\\
As the strength of our approach, we suggest the direct applicability of NSEs and their approximations to individual datasets
to potentially improve the quality of electrophysiological brain modelling. Thereby, 
the results of this study complement the recently developed statistical approaches
following from 1D NSEs \cite{moura2021anatomical,lahtinen2023silico}. Overall, this study advances the electrical conductivity approximation techniques applicable in electrophysiological modalities, where 
dynamic components affecting the conductivity atlases are
typically absent, e.g.,  EEG/MEG source localization  \cite{dannhauer2010}, tES \cite{herrmann2013transcranial}, and EIT \cite{cheney1999electrical,fernandez2011estimation,moura2021anatomical}. In particular, we have shown how to incorporate the dynamic blood flow effects in modelling the electrical conductivity when a high-resolution and high-intensity MRI segmentation with distinguishable blood vessels is available  \cite{fiederer2016}.
\\
We consider PPE an appropriate approximation of SE under the present modelling framework, since the MRI data does not allow a perfect segmentation of cerebral arteries. Hence, a more advanced solution based on NSEs might at least partly suffer from the limited accuracy of the segmentation. The current results suggest a spatial pressure variation between 80 and 115 mmHg, which matches with $\pm$5 mmHg discrepancy the normotensive systolic pressures found via numerical simulation in \cite{blanco2017blood} for arteries with diameter greater than 0.5 mm (e.g., 117 mmHg for 4.839 mm internal carotid artery, 113 mmHg for 3.448 mm basilar artery, 110 mmHg for 0.545 mm distal medial striate artery, and 85 mmHg for 1.039 mm posterior parietal branch of the middle cerebral artery).
\\
The velocity distribution can be considered appropriate based on experimental transcranial doppler ultrasound studies, e.g., \cite{gao2002optimal,kim2011blood}. A peak systolic velocity of 1.4 m/s has been suggested as a threshold criterion for mild stenosis in an intracranial vessel \cite{gao2002optimal}.  While our results reach that threshold, the vast majority of the velocity distribution stays below 1.4 m/s, i.e., in the range found for healthy subjects. While there are some obvious artifacts, the structure of the observed velocity distribution reflects the existing literature, as the greatest values were observed in the larger vessels, of which the middle and anterior cerebral arteries have overall greater velocities than the posterior cerebral artery or the vessels with a smaller diameter in the cortical branches \cite{ahn1991recording,caro2012mechanics}.
\\
The present model of excess blood flow and concentration near the arteries, resulting from arterial pressure, builds upon the utilization of Fick's law within the intricate network of microvessels \cite{berg2020modelling,arciero2017mathematical}. This model incorporates the assumption of a linear pressure drop  along the length of microvessels and considers the adaptability of microvessel diameter and length to regulate blood flow \cite{epp2020predicting}. Since venous vessels or any vessels outside the brain are difficult to distinguish based on MRI data \cite{fiederer2016}, we introduced a uniform sink term that covers the entire brain. Although directly validating the accuracy and reliability of numerical simulation results may be impractical, an observable correlation can be found between the estimated concentration distribution and whole-brain CBF scans obtained through MRI, positron emission tomography, and single-photon emission computed tomography \cite{chen2018evaluation, liu2018resting, taber2005blood}. A potential feature affecting the accuracy of our estimate obtained for the volumetric concentration is the fluid exchange between the microvessels and the tissue interstition, which was omitted in deriving the diffusion coefficient. We, however, deem that the uncertainty related to the sink term is likely to be a dominant error source as the venous flow rate is greater than the interstitial fluid exchange.
\\
Our excess concentration estimates and the general knowledge of perfusion scans both demonstrate that the excess blood in the brain is not limited to the large arteries but is somewhat spread in the neighborhood of those. Thus, the present distributions obtained using Archie's law and Hashin--Shtrikman bounds, which, in this study, were designed to take this aspect into account, might represent an improvement compared to the piecewise constant electrical conductivity estimates when a segmentation of the arteries is available. Furthermore, as suggested in \cite{fiederer2016}, an additional challenge arises due to a significant portion of the blood being limited to the subcortical region of the brain. This area is acknowledged to exhibit compromised localization and stimulation accuracy when utilizing non-invasive techniques such as EEG/MEG, EIT, and tES. We consider studying this aspect from the application point of view as important future work, while the focus of this study clearly was on establishing an appropriate modelling framework in order to take dynamic blood flow effects into account in an individualized MRI-based electrical conductivity atlas of the brain. The knowledge of the mixture models suggests \cite{j2001estimation} that while Hashin--Shtrikman bounds give valuable information about the potential modelling discrepancies, they do not achieve the accuracy of Archie's law, which is better suited for brain tissues. While the discrepancies between the models can be considered significant regarding the local value of electrical conductivity, the global differences observed in this study can be considered minor.
\\
As for the limitations of this study, based on the above reasoning, we do not expect that our model in its current form would be applicable for obtaining other than coarse estimates of blood pressure, velocity, volumetric concentration, and electrical conductivity distribution; the current results are limited to showing the feasibility of evaluating the PPE approximation for pressure and velocity to obtain estimates for concentration and electrical conductivity directly based on individualized data in electrophysiological modelling. As a governing limitation, we consider the weak distingushability of blood vessels from MRI data recorded with a magnetic flux density lower than 7T, as most datasets applied in electrophysiological head model generation comprise 3T or 4T measurements. Moreover, the present mathematical model is simplified and thus limited in its capability to approximate cerebral circulation.  Features omitted in this model include any time-dependencies of the blood flow, the contribution of viscoelastic arterial walls \cite{raghu2011comparative} and fluid exchange between microcirculation and tissue interstitium \cite{notaro2016mixed}. A simplified model is, however, well-motivated in this study due to the incompleteness of the MRI data with respect to obtaining blood vessel segmentation.
\\
\subsection{Future prospects}
While we achieved an appropriate overall match with the existing results, such as a similar range of values and distribution of the electrical conductivity perturbation as shown in \cite{moura2021anatomical,lahtinen2023silico}, further  studies are required to validate the present approach. 
The focus of future work will be on approximating the velocity field of cerebral circulation, which would necessitate solving a time-dependent system, e.g., the Navier--Stokes equation. This work will involve a complex and multidisciplinary study focusing on the electrical conductivity of the brain, incorporating mathematical modelling, numerical simulations, and data analysis to advance our understanding of brain function. We will use a segmentation of arterial blood vessels in the brain and a dynamic solution of NSEs coupled with Fick's law to represent microcirculation in a specific domain. For this purpose, we will solve a discretized non-Newtonian NSEs system using a two-stage process involving pressure estimation and velocity field updates. This includes regularization techniques to ensure numerical stability. The motivation behind this research lies in its significance for EEG, tES, and EIT, particularly, in scenarios characterized by dynamic modelling. 
\\
Looking ahead, the future path of this research will necessitate deep exploration of theoretical foundations, such as geometry, boundary conditions and viscosity models, as well as experimental validation.  An important aspect is, for example, how computing geometry influences the results obtained. To enlighten this, for example, a modelling study can be conducted to validate the performance of the current versions of PPE and Fick's law within a simplified computational geometry, such as a cylindrical domain. Furthermore, an experimental study can be conducted to compare the results of perfusion imaging with numerically simulated volumetric blood concentration.  We also aim at multi-subject studies and evaluations conducted at different scales.



\appendix

\section{Motivation for approximating NSEs}

\label{app:Diffusion model}
The Cauchy stress tensor, which is the symmetric component of the gradient of the velocity field, ${\bf u}$, is always split into two parts and written as $\sigma({\bf u}, p)=-{\bf I}\,p+\mu \,{\bf s}{\bf u}$, where {\bf I} is the unit tensor. ${\bf s}{\bf u}=\nabla {\bf u}+(\nabla {\bf u})^T$ and $\mu \,{\bf s}{\bf u}$ stand as a stress tensor and deformation (strain) rate tensor, respectively. Following is a definition of the diffusion force \cite{samavaki+tuomela}
\begin{equation}
 \begin{aligned}
{\bf L}{\bf u}:=\mathsf{div}({\bf s}{\bf u})=\Delta_B {\bf u}+\mathsf{grad}(\mathsf{div}({\bf u}))+{\bf Ri}({\bf u})\,,
\end{aligned}
\label{lu}
\end{equation}
where $\Delta_B$ is the Bochner Laplacian and ${\bf Ri}$ is the Ricci curvature \cite{samavaki_Ricci_2021}, which is given in the local coordinates by the Riemann curvature tensor ${\bf R}=R^h_{kij}$ as follows 
\begin{equation}
 \begin{aligned}
({\bf Ri}({\bf u}))_j:=({\bf Ri}:{\bf u})_j=R^k_{kij}g^{ih}u_h=R_{ij}u^i\,.
\end{aligned}
\label{ricci}
\end{equation}

\setcounter{lemma}{0}
\renewcommand{\thelemma}{\Alph{section}\arabic{lemma}}
\begin{lemma}
\[
 \mathsf{div}({\bf L}{\bf u})=2
  \mathsf{div}( \Delta_B {\bf u})=2\Delta_B( \mathsf{div}({\bf u}))+2
 \mathsf{div}({\bf Ri}({\bf u}))\,.
\]
\label{div-lu}
\end{lemma}
\begin{proof}
We have:
\[
 \mathsf{div}({\bf L}{\bf u})=\mathsf{div}(\mathsf{div}(\nabla {\bf u}))+\mathsf{div}(\mathsf{div}(\nabla {\bf u})^{T})=g^{hk}(u^i_{;hik}+u^i_{;hki})\,.
\]
By applying the \emph{Ricci identity} to formula \eqref{lu} we get
\[
\mathsf{div}({\bf L}{\bf u})=2(g^{hk}u^i_{;ihk}+g^{hk}u^i_{;k}R_{hi}+g^{hk}u^i R_{hi;k})\,.
\]
On the one hand, the following formula will follow easily  from definition
\[
\Delta_B(\mathsf{div} ({\bf u}))=\Delta_B (u^i_{;i})=\mathsf{div}(\mathsf{grad}(u^i_{;i}))=g^{hk}u^i_{;ihk}\,.
\]
On the other hand, by applying formula \eqref{ricci}, we obtain
\[
\mathsf{div}({\bf Ri} ({\bf u}))=\mathsf{div}(u^h R_{hi})=g^{hk}u^i_{;k}R_{hi}+g^{hk}u^i R_{hi;k}\,.
\]
Consequently, we can show
\begin{align*}
\mathsf{div}(\mathsf{div}(\nabla {\bf u}))&=\mathsf{div}(u^i_{;hi})=g^{hk}u^i_{,hik}
\\
&=g^{hk}(u^i_{;ihk}+u^i_{;k}R_{hi}+u^i R_{hi;k})\,,
\end{align*}
which proves $\mathsf{div}(\mathsf{div}(\nabla {\bf u}))=\frac{1}{2}\mathsf{div}({\bf L}{\bf u})$\,.
\end{proof}

\section{Incompressibility of flow with static pressure field}
\label{app:incompressibile_steady_state}

In this study, we assume that ${\bf u}_{,t}=0$. If $p = p({\bf x}; \cdot)$ is a static pressure distribution between $0$ and $t$, and ${\bf u}({\bf x};0) = {\bf u}_0$ with $\nabla \cdot {\bf u}_0 = 0$, the velocity field $\hat{\bf u} = {\bf u}- {\bf u}_0$ such that $\hat{\bf u}({\bf x};0) = 0$, can be obtained as follows
\[
   \hat{\bf u}({\bf x}; t) = \lim_{k \to \infty} \hat{\bf u}_{k} ({\bf x}; t)
\] 
where 
\[
\hat{\bf u}_{k}=({\bf I}+\Delta t \, \frac{\mu}{\rho} {\bf L} ) \, \hat{\bf u}_{k-1}+\hat{\bf u}_{1} \quad  \hbox{with} \quad \Delta t = \frac{t}{k},
\]
for $ k = 1, 2, \ldots$ and $\nabla \cdot \hat{\bf u}_k = 0$, which follows inductively from (\ref{pp}). Induction implies further that $\hat{\bf u}_{1} = {\rho}^{-1} \Delta t \, ({\rho} \, {\bf f} - \nabla p)$ and 
\[
\hat{\bf u}_{k} = \sum_{\ell = 1}^k ({\bf I}+\Delta t \, \frac{\mu}{\rho} {\bf L})^{\ell-1} \, \hat{\bf u}_1.
\]
Consequently, by the geometric series formula $({\bf I} - {\bf A})^{-1} = \sum_{\ell = 1}^\infty {\bf A}^{\ell -1}$ it holds that
{\setlength\arraycolsep{2pt} \begin{eqnarray}
\hat{\bf u} & = & \lim_{k \to \infty} \hat{\bf u}_{k} = \sum_{\ell = 1}^\infty ({\bf I}+\Delta t \, \frac{\mu}{\rho} {\bf L})^{\ell-1} \, \hat{\bf u}_1 \nonumber \\ & = & ({\bf I} - {\bf I} - \Delta t \, \frac{\mu}{\rho} {\bf L})^{-1} \hat{\bf u}_1  = - \frac{1}{\mu}  {\bf L}^{-1} \, ({\rho} \, {\bf f} - \nabla p)
\end{eqnarray}}
Substituting ${\bf u} = \hat{\bf u} + {\bf u}_0$, we have 
\[
{\bf u} = {\bf u}_0 + \frac{1}{\mu}  {\bf L}^{-1} \, (\nabla p - {\rho} \, {\bf f}) \quad \hbox{with} \quad \nabla \cdot {\bf u} = 0.
\]

\section{Length density of arterioles}

\label{app:arteriole_density}

The total microvessel count $\overline{\xi} = \overline{\xi}_a + \overline{\xi}_c + \overline{\xi}_v$ composed by the densities of  arterioles $\overline{\xi}_a$, capillaries $\overline{\xi}_c$  and venules $\overline{\xi}_v$ (Figure \ref{artery_arteriole_interface}), can be related to $\overline{\xi}_a$ based on the respective individual cross-sectional areas $A_a = \pi D_a^2 /4$, $A_c=\pi D_c^2 /4$ and $A_v = \pi D_v^2 /4$,  with $D_a$, $D_c$, $D_v$ denoting the diameters, and the relative fractions \cite{tu2015human} $\gamma_a$, $\gamma_c$ and $\gamma_v$, $\gamma_a + \gamma_c + \gamma_v = 1$ of the total area $A$ bound together via the following equation:
\begin{align*}
A  = \frac{A_a \overline{\xi}_a  }{\gamma_a} = \frac{A_c \overline{\xi}_c  }{\gamma_c} = \frac{A_v \overline{\xi}_v  }{\gamma_v}\,.
\end{align*}
It follows that 
\begin{align*}
\overline{\xi}_c = \frac{A_a \overline{\xi}_a \gamma_c}{A_c \gamma_a}\,, \quad \hbox{and} \quad \overline{\xi}_v = \frac{A_a \overline{\xi}_a \gamma_v}{A_v \gamma_a}\,,
\end{align*}
and, further, that
\begin{align*}
\overline{\xi}_a = \overline{\xi} \left( 1 + \frac{A_a \gamma_c}{A_c \gamma_a} + \frac{A_a \gamma_v}{A_v \gamma_a} \right)^{-1}\,.
\end{align*}



\section*{Acknowledgements}
The work of Maryam Samavaki and Sampsa Pursiainen is supported by the Academy of Finland Centre of Excellence (CoE) in Inverse Modelling and Imaging 2018–2025 (decision 336792) and project 336151; Yusuf Oluwatoki Yusuf was supported by the Magnus Ehrnrooth Foundation through the graduate student scholarship; Arash Zarrin Nia has been funded by a scholarship from the K. N. Toosi University of Technology; Santtu Söderholm's work has been funded by the ERA PerMED (PerEpi) project AoF 344712; Joonas Lahtinen's work has been funded by Väisälä Fund; Fernando Galaz Prieto's work has been funded by the ERA PerMed (PerEpi) project AoF 344712.



 \bibliographystyle{elsarticle-num} 
 \bibliography{cas-refs}





\end{document}